\documentclass[12pt,a4paper,oneside]{article}
\usepackage[utf8]{inputenc}
\usepackage[T1]{fontenc}
\usepackage{amsmath}
\usepackage[english]{babel}
\usepackage{amsthm,amssymb}
\usepackage{dsfont}
\usepackage{csquotes}
\usepackage{graphicx}
\usepackage{color, pgf}
\usepackage{hyperref}
\usepackage{url}
\usepackage{tikz}
\usepackage{booktabs}
\usepackage{multirow}
\usepackage{cancel}
\usetikzlibrary{calc,backgrounds,fit}
\usetikzlibrary{external}
\usetikzlibrary{patterns}
\usetikzlibrary{decorations.pathreplacing,decorations.markings,arrows}
\usetikzlibrary{shapes,intersections,positioning}
\usetikzlibrary{arrows}
\hypersetup{colorlinks=true, urlcolor=blue, citecolor=red} 
\newtheorem{thm}{Theorem}[section]
\newtheorem{defn}[thm]{Definition}
\newtheorem{lem}[thm]{Lemma}
\newtheorem{cor}[thm]{Corollary}
\newtheorem{rem}[thm]{Remark}
\newtheorem{prop}[thm]{Proposition}

\newcommand\centre[4][below]{%
  \node (#3) at #2 [circle,minimum size=0.5em,inner sep=0pt,fill] {};%
  \node [#1=0.15em] at (#3) {#4};}

\newcommand{\edgesint}{\mathcal{E}_{\operatorname{int}}}
\newcommand{\reg}{\operatorname{reg}(\mathcal{T})}
\newcommand{\Tau}{\mathcal{T}}
\newcommand{\ds}{d_{\sigma}}
\newcommand{\ms}{m_{\sigma}}

\newcommand{\E}{\mathbb{E}}
\newcommand{\N}{\mathbb{N}}
\newcommand{\R}{\mathbb{R}}

\newcommand{\eps}{\varepsilon}
\newcommand{\pe}{\psi_{\varepsilon}}
\newcommand{\ups}{u_{\varepsilon}}
\newcommand{\edges}{\mathcal{E}}
\newcommand{\ekint}{\mathcal{E}_{\operatorname{int}}^K}
\newcommand{\eint}{\mathcal{E}_{\operatorname{int}}}
\newcommand{\ek}{\mathcal{E}^K}

\newcommand{\un}{u^n}
\newcommand{\unm}{u^{n-1}}
\newcommand{\tun}{\tilde{u}^n}
\newcommand{\tunm}{\tilde{u}^{n-1}}

\newcommand{\dt}{\Delta t}
\newcommand{\tn}{t_n}

\newcommand{\erww}[1]{\mathbb{E}\left[#1\right]}

\newcommand{\Rtd}{\mathbb{R}^{d_h}}

\allowdisplaybreaks

\title{Study of a TPFA scheme for the stochastic Allen-Cahn problem with constraint through numerical experiments}
\author{Niklas Sapountzoglou\footnotemark[1] \and Aleksandra Zimmermann\footnotemark[1]}
\date{\today}

\begin{document}

\maketitle

\begin{abstract} 
We study the properties of a fully discrete numerical scheme for the stochastic Allen-Cahn problem with constraint on a bounded polygonal domain in two or three dimensions with homogeneous Neumann boundary condition. The scheme under consideration is of Two Point Flux Approximation (TPFA) type with respect to space and of semi-implicit Euler-Maruyama type with respect to time. The constraint is implemented by a multivalued, subdifferential operator which makes the problem a differential inclusion.
This operator is incorporated into the scheme via its Yosida approximation, at the expense of introducing an additional regularization parameter. 
From previous studies it is known that the time discretization of the Yosida approximation has to be implicit in order to obtain a convergent scheme. Consequently, a non-linear and non-smooth equation has to be solved at each time step and this is challenging for the computation of approximate solutions.
In this contribution, we introduce a splitting method to compute the solutions of the TPFA approximations. This method allows us to solve a linear equation in a first step. Then, a piecewise affine projection operator coming from the Yosida approximation can be computed explicitly in a second step. We quantify the error coming from the splitting method and show that the splitting method is accurate. We study properties of the scheme and provide convergence rates through numerical experiments. 

\end{abstract}

\footnotetext[1]{TU Clausthal, Institut f\"ur Mathematik, Clausthal-Zellerfeld, Germany, niklas.sapountzoglou@tu-clausthal.de, aleksandra.zimmermann@tu-clausthal.de}

\section{Introduction}\label{Section Introduction}
\subsection{The stochastic Allen-Cahn problem with constraint}
Let $T>0$ and $(\Omega,\mathcal{A},\mathds{P})$ a probability space endowed with a complete, right-continuous, filtration $(\mathcal{F}_t)_{t\geq 0}$. For $d\in \{2,3\}$, let $\Lambda$ be a bounded polygonal domain of $\R^d$ in the sense of \cite[Definition 1.47, p. 32]{Ern}.  We consider the fully discrete Euler-Maruyama/TPFA scheme for the following stochastic Allen-Cahn problem with constraint given in \cite{BSVZ25}:
\begin{align}\label{equation}
\begin{aligned}
du+(\psi-\Delta u)\,dt &=g(u)\,dW(t)+ (\beta(u) + f) \, dt, 
&&\text{ in }\Omega\times(0,T)\times\Lambda;\\
u(0,\cdot)&=u_0, &&\text{ in } \Omega\times\Lambda;\\
\nabla u\cdot \mathbf{n}&=0, &&\text{ on }\Omega\times(0,T)\times\partial\Lambda;
\end{aligned}
\end{align}
where $(W(t))_{t\geq 0}$ is a standard, one-dimensional Brownian motion with respect to $(\mathcal{F}_t)_{t\geq 0}$ on $(\Omega,\mathcal{A},\mathds{P})$, $\mathbf{n}$ denotes the unit normal vector to $\partial\Lambda$ outward to $\Lambda$ and $\psi\in \partial I_{[0,1]}(u)$, where
\begin{align*}
I_{[0,1]}(r)=\begin{cases} 0 & \text{if} \ r\in [0,1]\\
+\infty & \text{else,}
\end{cases}
\end{align*}
and the set-valued mapping $\partial I_{[0,1]}:[0,1] \rightarrow\mathcal{P}(\R)$, see, e.g., \cite[Chapter II]{brezis73}, is defined by
\begin{align}\label{251204_01}
\partial I_{[0,1]}(r)=\begin{cases} \{0\} & \text{if} \ r\in (0,1)\\
(-\infty,0] & \text{if} \ r=0\\
[0,\infty)& \text{if} \ r=1.
\end{cases}
\end{align}
We consider the following assumptions on the data:
\begin{itemize}
\item The initial datum $u_0\in L^2(\Omega;L^2(\Lambda))$ is $\mathcal{F}_0$-measurable and verifies $0\leq u_0(\omega,x) \leq 1$ for almost all $(\omega,x)\in \Omega\times \Lambda$.
\item The diffusion coefficient $g: \R\rightarrow\R$ is a Lipschitz-continuous function with Lipschitz constant $L_g\geq 0$ such that  $\operatorname{supp} g\subset [0,1]$.
\item $\beta: \R\rightarrow \R$ is a Lipschitz-continuous function with Lipschitz constant $L_\beta \geq 0$ such that (for convenience) $\beta(0)=0$.
\item $f\in L^2_{\mathcal{P}_T}(\Omega\times(0,T);L^2(\Lambda))$, the $L^2$-space of equivalence classes of predictable \footnote{$\mathcal{P}_{T}$ is called the predictable $\sigma$-field on $\Omega\times (0,T)$ (see \cite[p. 33]{LR}) defined as $\sigma(\{ (s,t]\times F_s \ | \ 0\leq s < t \leq T, \ F_s\in \mathcal{F}_s \} \cup \{\{0\}\times F_0 \ | \ F_0\in \mathcal{F}_0 \})$. } processes with values in $L^2(\Lambda)$. 
\end{itemize}
From the results of \cite{BBBLV17} it is well known that there exists a unique solution to \eqref{equation} in the following sense:
\begin{defn}\label{solution} 
Any pair of stochastic processes $(u,\psi)$ with
$\psi\in L^2_{\mathcal{P}_T}(\Omega\times(0,T);L^2(\Lambda))$ and 
\begin{align*}
 u\in L^2(\Omega;\mathcal{C}([0,T];L^2(\Lambda)))\cap L^2_{\mathcal{P}_T}(\Omega\times(0,T);H^1(\Lambda)),
\end{align*}
is a solution to Problem \eqref{equation} if $0\leq u \leq 1$, $\psi\in \partial I_{[0,1]}(u)$ almost everywhere in $\Omega\times (0,T)\times\Lambda$, and if $(u,\psi)$ satisfies the equation
\begin{align}\label{eq 161225_01}
u(t) -u_0+\int_0^t (\psi(s)-\Delta u(s)) \,ds=\int_0^t g(u(s))\,dW(s)+ \int_0^t (\beta(u(s)) + f(s)) \, ds,
\end{align}
for all $t\in [0,T]$ in $L^2(\Lambda)$, $\mathds{P}$-a.s in $\Omega$, where $\Delta$ denotes the Laplace operator on $H^1(\Lambda)$ associated with the weak Neumann boundary condition.
\end{defn}
\subsection{State of the art}
The study of numerical schemes for stochastic partial differential equations (SPDEs) has been a
very active subject in recent decades and for this reason, an extensive literature on this topic is available. We refer the interested reader to \cite{ACQS20,DP09,OPW20} for a general overview and to the references therein. Concerning the numerical analysis of variational solutions, in the past the use of finite-element methods has been favored and extensively employed (see, e.g., \cite{BBNP14, BHL21} for a detailed exposition of existing papers). In recent years, there has been increasing interest in numerical approximations which preserve the specific structure of the underlying equations. In this way, important physical properties of the underlying model are implemented at the numerical level, and the robustness and the stability of numerical methods are improved. As pointed out in \cite{D14}, among the numerous families of numerical methods, e.g., finite difference, finite element, discontinuous Galerkin, Gradient Discretization Method, finite-volume methods are well-suited for applications in which the conservation of quantities is important. For finite-volume methods, balance and local conservativity are the leading principles in the construction of numerical fluxes. The Two Point Flux Approximation (TPFA) scheme is a cell-centered finite-volume method which is particularly cheap to implement, since its matrices are very sparse. In the last decade, there has been growing interest in the use of finite-volume methods for the spatial discretization of SPDEs. First, monotone schemes have been derived for the approximation of hyperbolic problems perturbed by multiplicative noise, known as stochastic first order scalar conservation laws in the literature, see, e.g., \cite{BCG162, BCG161, BCC20, BCG17}. More recently, the convergence analysis of a numerical scheme of Euler-Maruyama type with respect to time and of TPFA-type with respect to space for a stochastic heat equation with multiplicative Lipschitz noise has been studied in \cite{BNSZ22}. In this contribution, convergence of approximate solutions was proved by adapting the stochastic compactness method based on Skorokhod’s representation theorem. In \cite{BNSZ23,BSZ25}, these results have been extended to more general stochastic PDEs of diffusion-convection type where the convective term is given by a first-order divergence type operator. At the same time, the authors avoided the use of the stochastic compactness argument. In \cite{SZ25}, explicit error bounds for the numerical scheme in \cite{BNSZ22} have been developed and complemented with computational experiments. Strong error estimates for a finite element scheme of the stochastic Allen-Cahn equation with the classical double-well potential $\frac{1}{4}(u^2-1)^2$ have been studied in \cite{MP16}. Recently, the authors in \cite{ZZZC26} have shown weak error estimates for a Galerkin discretization in space and an Euler discretization in time considering the same equation.
Finally, in \cite{BSVZ25}, the methods developed in \cite{BNSZ22} and \cite{BSZ25} have been applied to a stochastic Allen–Cahn problem with constraint.  
Stochastic Allen-Cahn problems with constraint as considered in \cite{BBBLV17} arise in models for the evolution of damage over time at an interface of a composite material. More precisely, one considers the stochastic differential inclusion 
\begin{align}\label{ACst}
{\mathrm d} u -\Delta u\,{\mathrm d}t + \partial I_{[0,1]}(u)\,{\mathrm d}t &\ni g(u)\, {\mathrm d}W + (\beta(u)+f) \,{\mathrm d}t 
\end{align}
on a bounded domain of $\mathbb{R}^d$ with homogeneous Neumann boundary conditions, where $\partial I_{[0,1]}$ is defined according to \eqref{251204_01}. 
In \eqref{ACst}, the solution $u$ represents the local proportion of intact bonds within the considered material which is a quantity between $0$ and $1$. If $u=0$, the material is totally damaged, if $u=1$, the material is without damage and, if $0<u<1$, the material is partially damaged. 
The physical constraint on the solution $u$ is ensured by the presence of the multivalued maximal monotone operator $\partial I_{[0,1]}$. The addition of a stochastic perturbation to this Allen-Cahn model is motivated by the consideration of changes of the material structure at a microscopic scale such as the formation of cavities during damage, the random distribution of micro-cracks, and external forces.  In \cite{BSVZ25} a finite-volume scheme for the stochastic Allen-Cahn problem with constraint \eqref{ACst} is proposed in the case of a polygonal domain and for the space dimensions $d=2$ or $d=3$.
 Moreover, the convergence of the scheme towards the variational solution is studied. The main feature of the proposed scheme is the approximation of the set-valued constraint by its Yosida approximation indexed by $\eps>0$. In this way, the additional parameter $\eps>0$ enters the discrete equation and the passage to the limit is performed simultaneously with the time, space and the regularization parameter. The authors show the convergence of the proposed Euler-Maruyama-TPFA-$\eps$ type scheme under a coupling condition of the time parameter and the parameter related to the Yosida approximation of the set-valued constraint. 
\subsection{Discretization framework and notation}\label{sectiontwo}
In the following, we introduce the general notation and recall important definitions related to temporal and spatial discretizations. Several of these notions already appear in
\cite{BNSZ22,BNSZ23,BSZ25,SZ25}, but we include them here for the reader’s convenience.

\subsubsection{General notation}

\begin{itemize}
\item[$\bullet$] The integral over $\Omega$ with respect to the probability measure $\mathds{P}$ is denoted by $\E[\cdot]$, and is called the expectation.
\item[$\bullet$] For any $x,y$ in $\R^d$, the euclidean norm of $x$ is denoted by $|x|$, and the associated scalar product of $x$ and $y$ by $x\cdot y$.
\item[$\bullet$] The $d$-dimensional Lebesgue measure of $\Lambda$ is denoted by $|\Lambda|$, this reuses the absolute value symbol. 
\item[$\bullet$]The $d$-dimensional identity mapping will be denoted by $I$, i.e., $I(x)=x$ for all $x\in\mathbb{R}^d$.
\end{itemize}

\subsubsection{Uniform time step and admissible finite-volume meshes}\label{mesh}
We refer to the Euler-Maruyama-TPFA-$\eps$ scheme proposed in \cite{BSVZ25} as TPFA-scheme. Let us introduce the temporal and spatial discretizations:
The temporal one is achieved using a uniform subdivision: setting $N\in \mathbb{N}$, the fixed time step is defined by $\dt=\frac{T}{N}$ and the interval $[0,T]$ is decomposed in $0=t_0<t_1<...<t_N=T$ equidistantly with $\tn=n \dt$ for all $n\in \{0, ..., N\}$. For the spatial one, following \cite[Definition 9.1]{EymardGallouetHerbinBook}, we consider admissible finite-volume meshes as defined hereafter: 
\begin{defn} (Admissible finite-volume mesh)\label{defmesh} 
An admissible finite-volume mesh of $\Lambda$, denoted by $(\Tau,\edges,\mathcal{P})$, consists of a finite number of so-called control volumes $K$, which are open polygonal convex subsets of $\Lambda$, denoted by $\Tau$, a finite number of subsets $\sigma\in \overline{\Lambda}$ contained in hyperplanes of $\R^d$, denoted by $\edges$ \footnotemark[1]\footnotetext[1]{These $\sigma\in\mathcal{E}$ will by called synonymously edges or interfaces of the control volumes.}, with strictly positive $(d-1)$-dimensional Lebesgue measure, and a finite sequence $\mathcal{P}=(x_K)_{K\in\Tau}$ of points in $\Lambda$, also called centers, satisfying the following properties
\begin{itemize}
\item $\overline{\Lambda}=\bigcup_{K\in\Tau}\overline{K}$.
\item For any $K\in\mathcal{T}$, there exists a subset $\ek$ of $\edges$ such that $\partial K=\overline{K}\setminus K=\bigcup_{\sigma\in \ek}\overline{\sigma}$ and $\edges=\bigcup_{K\in\mathcal{T}}\ek$. $\ek$ is called the set of edges or interfaces of $K$.
\item For any $K,L\in\Tau$, with $K\neq L$ either the $(d-1)$ Lebesgue measure of $\overline{K}\cap \overline{L}$ is $0$ or $\overline{K}\cap \overline{L}=\overline{\sigma}$ for some $\sigma\in\edges$. In the second case we will write $\sigma=K|L$ or $\sigma=L|K$.
\item The sequence $\mathcal{P}=(x_K)_{K\in\mathcal{T}}$ is such that $x_K\in \overline{K}$ for all $K\in\mathcal{T}$ and it is assumed that the straight line between $x_K$ and $x_L$ is orthogonal to $\sigma=K|L$.
\item For any $\sigma\in\edges$ such that $\sigma\subset\partial\Lambda$, let $K$ be the control volume such that $\sigma \in\ek$. If $x_K\notin\sigma$, the straight line going through $x_K$ and orthogonal to $\sigma$ has a nonempty intersection with $\sigma$.
\end{itemize}
\end{defn}

\begin{figure}[htbp!]
\centering
\begin{tikzpicture}[scale=3]

  \clip (-1.4,-0.8) rectangle (2.0,1.7);

    \foreach \P/\x/\y in {
        A/-1/0.6,
        B/0/1.2,
        C/0/-0.2,
        D/1.5/0.75}
        \node[circle,fill,inner sep=1.2pt] (\P) at (\x,\y) {};

  \centre[above right]{(-0.6,0.5)}{xK}{$x_K$};
  \centre[above left]{(0.9,0.5)}{xL}{$x_L$};

  \draw[thin] (A)--(B)--(C)--(A);
  \draw[thin] (D)--(B)--(C)--(D);

  \draw[very thick] (B)--(C);
  \node[below right=4pt]
    at ($(B)!0.5!(C)$) {$\sigma = K|L$};

  \node at (-0.7,0.15) {\large $K$};
  \node at (0.85,0.05) {\large $L$};

  \draw[dashed] (xK)--(xL);
    \coordinate (KK) at ($(xK)!0.52!90:(xL)$);
    \coordinate (LL) at ($(xL)!0.52!-90:(xK)$);
    
    \draw[dotted] (xK)--(KK);
    \draw[dotted] (xL)--(LL);
    
    \draw[|<->|]
      (KK)--(LL)
      node[midway,above=3pt,fill=white] {$\ds$};
  \coordinate (nkl) at ($(B)!0.4!(C)$);
  \draw[->,>=latex,thick]
    (nkl)--($(nkl)!0.35cm!90:(C)$)
    node[above] {$\mathbf n_{K,\sigma}$};

  \draw[dashed]
    (xK)--($(xK)!2!( $(A)!(xK)!(B)$ )$);
  \draw[dashed]
    (xK)--($(xK)!2!( $(A)!(xK)!(C)$ )$);

  \draw[dashed]
    (xL)--($(xL)!2!( $(D)!(xL)!(B)$ )$);
  \draw[dashed]
    (xL)--($(xL)!2!( $(D)!(xL)!(C)$ )$);

\end{tikzpicture}
\caption{Two neighboring control volumes $K,L \in \Tau$ with interface $\sigma = K|L$ in $\R^2$}
\end{figure}

For a given admissible finite-volume mesh $(\Tau,\edges,\mathcal{P})$ of $\Lambda$, the following associated notation will be used in the rest of the paper.\\

\textbf{Notation.}
\begin{itemize}

\item The mesh size is denoted by $h=\operatorname{size}(\Tau)=\sup\{\operatorname{diam}(K): K\in\Tau\}$.

\item The number of control volumes $K\in\Tau$ is denoted by $d_h\in\mathbb{N}$, where $h=\operatorname{size}(\Tau)$.

\item The set of interior 
interfaces is denoted by
$\edgesint:=\{\sigma\in\edges:\sigma\nsubseteq \partial\Lambda\}$.
Moreover we set $\ekint := \ek \cap \eint$ for any $K \in \Tau$.

\item For any $K\in\Tau$ fixed, the set of neighboring control volumes to $K$ will be denoted by $\mathcal{N}(K)$.

\item For any $K\in\Tau$, the $d$-dimensional Lebesgue measure of $K$ is denoted by $m_K$.

\item For any $\sigma\in\edgesint$, the $(d-1)$-dimensional Lebesgue measure of $\sigma$ is denoted by $m_\sigma$.

\item For any $K\in\Tau$ and any $\sigma \in \ek$, the unit vector on $\sigma$ pointing out of $K$ is denoted by $\mathbf{n}_{K,\sigma}$.

\item For any neighboring control volumes $K,L\in\Tau$ the euclidean distance between $x_K$ and $x_L$ is denoted by $d(x_K,x_L)$ and will be abbreviated by $\ds$, where $\sigma=K|L$ is the corresponding interface, if necessary. More generally, for any $x\in\mathbb{R}^d$ and any set $A\subset\mathbb{R}^d$ the euclidean distance between $x$ and $A$ will be denoted by $d(x,A)$.

\item We introduce a positive number 
\begin{align*}
\reg = \max_{\genfrac{}{}{0pt}{}{\sigma\in\ek}{\scriptscriptstyle K\in\mathcal{T}}}\frac{\operatorname{diam}(K)}{d(x_K,\sigma)}
\end{align*}
that may a-priori depend on the mesh size $h>0$.
\end{itemize}

\begin{rem}
A wide range of meshes meet the above assumptions, e.g., triangular meshes in the sense of \cite[Example 9.1]{EymardGallouetHerbinBook} and some classes of Vorono{\"{i}} meshes, as described in \cite[Example 9.2]{EymardGallouetHerbinBook}. For more details on mesh generation, we refer to \cite[Subsection 1.3 and Section 7]{Ern}.
\end{rem}

\subsection{Aim of this study}
From the point of view of applications, some questions remain open for the TPFA-scheme developed in \cite{BSVZ25}. The multivalued term in \eqref{equation} ensures that, for initial data with values in $[0,1]$, the values of the solution $u$ stay in $[0,1]$. It would be desirable to know if the numerical approximations given by the scheme also take values in $[0,1]$ and how the parameter $\theta>0$ from the coupling condition $\Delta t=\mathcal{O}(\eps^{2+\theta})$ between the time discretization parameter $\Delta t$ and the regularization parameter $\eps>0$ stated in \cite{BSVZ25} may affect the behaviour of numerical solutions. Secondly, to provide enough regularity on the approximate solutions and to obtain the necessary a-priori estimates, the TPFA-scheme in \cite{BSVZ25} has to be time implicit in the Yosida approximation. Consequently, one has to solve nonlinear equations using appropriate numerical methods for each time step. Since the Yosida approximation is given by a sequence of piecewise affine functions which are depending on the discretization parameter, this issue is delicate. A numerical method for solving the TPFA-scheme has not been provided in \cite{BSVZ25}.
In this contribution, we address these questions, by introducing a splitting method for the calculation of the solutions of the TPFA-scheme, and providing numerical simulations. More precisely, we consider the following family of SPDEs depending on a parameter $\eps>0$:
\begin{align}\label{eqeps}
\begin{aligned}
d\ups+(\pe(\ups)-\Delta \ups)\,dt &=g(\ups)\,dW(t) + (\beta(\ups) + f) \, dt, &&\text{ in }\Omega\times(0,T)\times\Lambda;\\
\ups(0,\cdot)&=u_0, &&\text{ in } \Omega\times\Lambda;\\
\nabla \ups\cdot \mathbf{n}&=0, &&\text{ on }\Omega\times(0,T)\times\partial\Lambda;
\end{aligned}
\end{align}
where $\pe:\R\rightarrow \R$ denotes the Yosida approximation of $\partial I_{[0,1]}$ (see, \textit{e.g.,} \cite{barbu76,brezis73}), defined for all $r\in\R$ by
\begin{eqnarray}\label{SPPR}
\pe(r)=-\frac{(r)^-}{\eps}+\frac{(r-1)^{+}}{\eps}=
\begin{cases}
\frac{r}{\eps}, &r\leq 0\\
0, &r\in(0,1)\\
\frac{r-1}{\eps}, &r\geq 1.
\end{cases}
\end{eqnarray}
Taking \eqref{eqeps} as a starting point, the following TPFA-scheme has been proposed in \cite{BSVZ25}:\\
Let $\Tau$ be an admissible mesh in the sense of Definition \ref{defmesh}  with mesh size $h>0$. For a $\mathcal{F}_0$-measurable random element $u_0\in L^2(\Omega;L^2(\Lambda))$ we define $\mathbb{P}$-a.s. in $\Omega$ its piecewise constant spatial discretization $u_h^0= \sum_{K \in \Tau} u_K^0 \mathds{1}_K$, where
\begin{align*}
u_K^0= \frac{1}{m_K} \int_K u_0(x) \, dx, ~~\forall K \in \mathcal{T}.
\end{align*}
Now, for fixed $N \in \N$ and $\eps >0$, for any $n \in \{1,\dots,N\}$, assuming the random vector $u^{n-1}= (u^{n-1}_K)_{K\in\Tau} \in \mathbb{R}^{d_h}$ to be given, the random vector $u^n=(u^n_K)_{K\in\Tau}\in\mathbb{R}^{d_h}$ is defined to be the solution of the following equation:
\begin{align}\label{equationapprox}
\begin{split}
&\frac{m_K}{\Delta t}(u_K^n-u_K^{n-1}) +\sum_{\sigma=K|L\in\ekint}\frac{m_\sigma}{\ds}(u_K^n-u_L^n)+m_K\pe(u_K^n)\\
&=\frac{m_K}{\Delta t}g(u_K^{n-1})\Delta_n W + m_K \beta(u_K^n) + m_K f_K^{n-1},
\end{split}
\end{align}
for each $K\in \Tau$, $\mathbb{P}$-a.s in $\Omega$, where $\Delta_n W:= W(t_n) - W(t_{n-1})$ denotes the increment of the Brownian motion between $t_{n}$ and $t_{n-1}$ and
\begin{align*}
f_K^{n-1} = \frac{1}{m_K \Delta t} \int_{t_{n-1}}^{t_n} \int_K f(x,t) \, dx \, dt.
\end{align*}
We remark that the solution $u^n$ of \eqref{equationapprox} depends on the parameters $N\in\mathbb{N}$, $h>0$ and $\eps>0$, i.e., $u^n=u_{h,N, \eps}^n$. To avoid overloading the notation, we will omit them in the notation for solutions of \eqref{equationapprox} whenever they are fixed. The following results are known:

\begin{prop}\cite[Proposition 3.2]{BSVZ25}\label{251205_p1}
 Let $\Tau$ be an admissible finite-volume mesh of $\Lambda$ in the sense of Definition \ref{defmesh} with mesh size $h>0$. Moreover, let an initial $\mathcal{F}_0$-measurable random vector $u_h^0= (u^0_K)_{K\in\Tau}\in\mathbb{R}^{d_h}$, $N\in\mathbb{N}$, and $\eps>0$ be given. Then, there exists a unique solution $(u^n)_{1\le n \le N} \in (\mathbb{R}^{d_h})^N$ to \eqref{equationapprox}. Moreover, for any $n\in \{1,\ldots,N\}$, $u^n$ is a $\mathcal{F}_{t_n}$-measurable random vector.
\end{prop}

\begin{thm}\cite[Theorem 1.5]{BSVZ25}\label{Theorem 020626_02}
Let $(u,\psi)$ be the unique solution of Problem \eqref{equation} in the sense of Definition \ref{solution}. 
Let $(\Tau_m)_{m\in \mathbb{N}}$ be a sequence of admissible finite-volume meshes of $\Lambda$ in the sense of Definition \ref{defmesh} such that the sequence of their mesh sizes $(h_m)_{m\in \mathbb{N}}$ tends to $0$ and there exists a constant $C_{\operatorname{reg}}>0$ such that $\operatorname{reg}(\Tau_m)\leq C_{\operatorname{reg}}$ for all $m\in\mathbb{N}$. Let $(N_m)_{m\in \mathbb{N}}\subset \mathbb{N}$ be a sequence which tends to infinity and let $(\eps_m)_{m\in\mathbb{N}}$ be a sequence of positive real numbers such that $\lim_{m\rightarrow+\infty} \eps_m=0$. For $m\in\mathbb{N}$, let $u_{h_m,N_m,\eps_m}$ be the piecewise constant function induced by the solutions of \eqref{equationapprox} according to Proposition \ref{251205_p1}.
If there exists $\theta>0$ such that, for any $m\in \mathbb{N}$, $\frac{T}{N_m}=\mathcal{O}((\eps_m)^{2+\theta})$, then
the sequence $(u_{h_m,N_m,\eps_m})_{m\in \mathbb{N}}$ converges to $u$ and the sequence $(\psi_{\eps_m}(u_{h_m,N_m,\eps_m}))_{m\in \mathbb{N}}$ converges to $\psi$, strongly and weakly in $L^2(\Omega;L^2(0,T;L^2(\Lambda)))$, respectively.
\end{thm}

Our aim is to propose a modification of the TPFA-scheme in \eqref{equationapprox} which is suitable for the numerical implementation of the nonlinear term arising from the Yosida approximation. In Section \ref{splitting method}, we will introduce the new scheme which we call the splitting method. Then we will provide estimates on the error between the approximate solutions given by the splitting method and the approximate solutions given by the TPFA-scheme \eqref{equationapprox}. In Section \ref{properties}, we collect further properties of the splitting method. The projection operator is studied in Subsection \ref{projection}. In Subsection \ref{noforce} we study the properties of \eqref{eqeps} and the splitting method in the case of no external forcing, i.e., in the case $\beta=0$ and $f=0$. In particular, we establish a maximum principle for \eqref{eqeps} and a sign condition for the splitting method in the case of no external forcing. Section \ref{num} is devoted to numerical experiments. We validate our theoretical findings on the projection operator from Subsection \ref{projection} and study numerical error estimates for the splitting method.
\section{Splitting method}\label{splitting method}
\subsection{Derivation of the splitting method}
We fix $N \in \N$, $\eps>0$, and an admissible mesh $\Tau$ in the sense of Definition \ref{defmesh} with mesh size $h>0$. For a vector $w=(w_K)_{K \in \Tau} \in \Rtd$ and any function $\phi: \R \to \R$ we will use the notation $\phi(w) := (\phi(w_K))_{K \in \Tau}$. In this way we may consider $\pe$ as a non-linear function from $\Rtd$ to $\Rtd$. For a  fixed $n\in\{1,\ldots,N\}$, let $\unm = (\unm_K)_{K \in \Tau} \in \Rtd$ and $f^{n-1}=(f^{n-1}_K)_{K \in \Tau} \in \Rtd$ be given. Then, the solution $\un = (\un_K)_{K \in \Tau} \in \Rtd$ of \eqref{equationapprox} may be rewritten as
\begin{align}\label{eq_011025_01}
    (M + \Delta t A)\un + \Delta t M (\pe(\un) - \beta(\un))= M (\unm + g(\unm) \Delta_n W + \Delta t f^{n-1}),
\end{align}
with the following matrices in $\R^{d_h \times d_h}$:  
\begin{align}\label{260709_01}
M = \operatorname{diag}((m_K)_{K \in \Tau})
\end{align}
 and 
 \begin{align}\label{260709_02}
 \begin{aligned}
 A = (a_{K,L})_{K,L \in \Tau}=
    \begin{cases}
        a_{K,K} &= \sum\limits_{\sigma \in \ekint} \frac{\ms}{\ds}, ~K \in \Tau \\
        a_{K,L} &= - \frac{\ms}{\ds},  ~~~~~~K \in \Tau, L \in \mathcal{N}(K), \sigma= K|L \\
        a_{K,L} &=0,  ~~~~~~~~~~\text{ else}.
    \end{cases}
    \end{aligned}
\end{align}
By construction, $A$ is sparse and symmetric such that $\sum\limits_{L \in \Tau} a_{K,L} =0$ for any $K \in \Tau$. 
Since, for any $v\in\mathbb{R}^{d_h}$, we have
\[v^TAv=\sum\limits_{\sigma \in \eint} \frac{\ms}{\ds}(v_K-v_L)^2\geq 0,\]
it follows that $A$ is positive semi-definite w.r.t. the euclidean scalar product on $\R^{d_h}$. From \eqref{eq_011025_01} it follows that
\begin{align}\label{eq_011025_02}
    u^n = (M + \Delta t A + \Delta t M (\pe -\beta))^{-1} M (\unm + g(\unm) \Delta_n W + \Delta t f^{n-1}).
\end{align}
Due to the fact that $\pe - \beta$ is non-linear and $M+ \Delta t A$ is not a diagonal matrix, we know from the abstract theory that the solution $u^n$ exists. However, no explicit solution formula is known in general and we cannot calculate the nonlinear inverse mapping of $(M + \Delta t A + \Delta t M (\pe -\beta))$ on a computer exactly. To address this fact, first we split the calculation of $\un$ into two steps:
\begin{itemize}
\item[Step 1.] Let $\unm \in \Rtd$ be given. Then we define $\hat{u}^n \in \Rtd$ by
\begin{align}\label{eq_011025_03}
\hat{u}^n := (M + \Delta t A)^{-1} M (\unm + g(\unm) \Delta_n W + \Delta t f^{n-1}).
\end{align}
\item[Step 2:] For $\hat{u}^n$ given by \eqref{eq_011025_03}, we obtain
\begin{align*}
u^n=\left((M+\Delta t A)+\Delta tM(\pe-\beta)\right)^{-1}(M+\Delta tA)\hat{u}^n,
\end{align*}
or, equivalently,
\begin{align}\label{eq_011025_04}
u^n + (M+ \Delta t A)^{-1} \Delta t M \pe (\un) = \hat{u}^n + (M+ \Delta t A)^{-1} \Delta t M \beta(\un).
\end{align}
In order to solve \eqref{eq_011025_04} approximately we discard the term $\Delta t A$ on the left-hand side of \eqref{eq_011025_04}, replace the time-implicit term $\beta(u^n)$ by its explicit counterpart $\beta(u^{n-1})$ on the right-hand side of \eqref{eq_011025_04}. In practice we form a recursion by replacing the unknown $\unm$ by the available iterate $\tunm$ and solve
\begin{align}\label{eq_011025_05}
\begin{aligned}
&\tun + \Delta t \pe(\tun) = \hat{u}^n + (M+ \Delta t A)^{-1}\Delta t M\beta(\tunm) \\
&= (M + \Delta t A)^{-1} M \left(\tunm + g(\tunm) \Delta_n W + \Delta t f^{n-1} + \Delta t \beta(\tunm) \right) \\
&=: v^n
\end{aligned}
\end{align}
for $\tun \in \Rtd$. 
After multiplication with $M + \Delta t A$, this introduces a consistency term $(\Delta t)^2 A\pe(\tun)$ whose effect we quantify in Theorem \ref{Theorem 020626_01}.
Since $\pe$ acts on each component separately, we may rewrite \eqref{eq_011025_05} as
\begin{align}\label{eq_011025_06}
\tun_K + \Delta t \pe(\tun_K) = v^n_K \text{ for each } K \in \Tau.
\end{align}
Now, since $\pe$ is monotone, $I + \Delta t \pe$ is strictly monotone and therefore invertible.
Consequently, the solution of \eqref{eq_011025_06} is given by
\begin{align}\label{eq_011025_07}
\tun_K = (I+ \Delta t \pe)^{-1}(v^n_K) \text{ for each } K \in \Tau.
\end{align}
Thanks to the particular piecewise linear structure of $\pe$, the function $(I + \Delta t \pe)^{-1}:\mathbb{R}\rightarrow\mathbb{R}$ may be calculated explicitly with
\begin{align*}
(I+ \Delta t \pe)^{-1}(r) = \begin{cases}
\frac{\eps r}{\eps + \Delta t}, &r<0 \\
r, &0 \leq r \leq 1 \\
\frac{\eps r + \Delta t}{\eps + \Delta t}, &r >1.
\end{cases}
\end{align*}
Hence, $\tun_K$ is given by
\begin{align*}
\tun_K = \begin{cases}
\frac{\eps v^n_K}{\eps + \Delta t}, &v^n_K<0 \\
v^n_K, &0 \leq v^n_K \leq 1 \\
\frac{\eps v^n_K + \Delta t}{\eps + \Delta t}, & v^n_K >1.
\end{cases}
\end{align*}
\end{itemize}
\begin{defn}\label{splittingmethod}
Let $u_0 \in \Rtd$ and $f^{n-1}\in \Rtd$ for all $n=1,\ldots,N$ be given. For $M, A\in \R^{d_h \times d_h}$ defined in \eqref{260709_01} and \eqref{260709_02}, respectively, the finite sequence $\tun = (\tun_K)_{K \in \Tau} \in \Rtd$, $n=1,\ldots,N$, given by the recursion formula $\tilde{u}^0 = u^0$,
\begin{align*}
\tun = (I + \Delta t \pe)^{-1} (M + \Delta t A)^{-1} M(\tunm + g(\tunm) \Delta_n W + \Delta t f^{n-1} + \Delta t \beta(\tunm)), 
\end{align*}
$n \in \{1,...,N\}$, is called the solution of the splitting method.
\end{defn}
\begin{rem}
According to the results in \cite{BSVZ25}, we may not replace the time-implicit discretization of $\pe$ in \eqref{eq_011025_04} by its time-explicit counterpart, since this time-implicit discretization is essential for the convergence and well-posedness of the scheme.
\end{rem}
\subsection{Error estimates}
The splitting method introduced in Definition \ref{splittingmethod} is an approximation procedure derived from the TPFA scheme in matrix notation given in \eqref{eq_011025_01}. In this subsection, we will focus on the error between the solution of the splitting method and the solution to the TPFA scheme.

\begin{lem}\label{Lemma 290526_01}
Let $u^0 \in \Rtd$ be given. For $N\in\mathbb{N}$, $\eps >0$, and an admissible mesh $\Tau$ in the sense of Definition \ref{defmesh} with mesh size $h>0$, let $\tun$ be the solution of the splitting method given in Definition \ref{splittingmethod}, with respect to the initial datum $u^0$. Then, there exists a constant $K_1>0$ only depending on $L_g$, $L_{\beta}$, $T$ and $u_0$ such that
\begin{align*}
\sup\limits_{n \in \{1,...,N\}} \erww{\Vert \tun \Vert_{L^2(\Lambda)}^2} + \sum\limits_{n=1}^N \erww{\Vert \tun -\tunm \Vert_{L^2(\Lambda)}^2} + \Delta t \sum\limits_{n=1}^N \erww{| \tun |_{1,h}^2} \leq K_1,
\end{align*}
for all $N\in\mathbb{N}$ large enough, where for $v=(v_K)_{K \in \Tau}$ we identify $v$ with the simple function $\sum_{K \in \Tau} v_K \mathds{1}_K\in L^2(\Lambda)$ and we set
\begin{align*}
|v|_{1,h}^2 := \sum\limits_{\sigma \in \eint} \frac{\ms}{\ds} (v_K - v_L)^2.
\end{align*}
\end{lem}

\begin{proof}
Since $\tun$ is the solution of the splitting method, we have
\begin{align}\label{eq 290526_01}
\begin{aligned}
&M(\tun - \tunm) + \Delta t A \tun + \Delta t M \pe(\tun) + (\Delta t)^2 A \pe(\tun) \\
&= M(g(\tunm)\Delta_n W + \Delta t f^{n-1} + \Delta t \beta(\tunm)).
\end{aligned}
\end{align}
for all $n=1,\ldots,N$. Taking the euclidian scalar product with $\tun=(\tun_K)_{K\in\Tau}$ in \eqref{eq 290526_01} we arrive at:
\begin{align}\label{260716_01}
\begin{aligned}
&\sum\limits_{K \in \Tau} m_K (\tun_K - \tunm_K)\tun_K + \Delta t \sum\limits_{K \in \Tau} \sum\limits_{\sigma \in \ekint} \frac{\ms}{\ds} (\tun_K - \tun_L)\tun_K \\
&+ \Delta t \sum\limits_{K \in \Tau} m_K \pe(\tun_K) \tun_K + (\Delta t)^2 \sum\limits_{K \in \Tau} \sum\limits_{\sigma \in \ekint}  \frac{\ms}{\ds} (\pe(\tun_K) - \pe(\tun_L))\tun_K \\
&= \sum\limits_{K \in \Tau} m_K g(\tunm_K) \tun_K \Delta_n W + \Delta t \sum\limits_{K \in \Tau} m_K f_K^{n-1} \tun_K + \Delta t \sum\limits_{K \in \Tau} m_K \beta(\tunm_K) \tun_K.
\end{aligned}
\end{align}
We have the discrete integration by parts identity
\begin{align*}
\sum\limits_{K \in \Tau} \sum\limits_{\sigma \in \ekint} \frac{\ms}{\ds} (\tun_K - \tun_L)\tun_K = \sum\limits_{\sigma \in \eint} \frac{\ms}{\ds} (\tun_K - \tun_L)^2,
\end{align*}
hence the monotonicity of $\pe$ yields
\begin{align*}
(\Delta t)^2 A \pe(\tun) \cdot \tun &= (\Delta t)^2 \sum\limits_{K \in \Tau} \sum\limits_{\sigma \in \ekint} \frac{\ms}{\ds} (\pe(\tun_K) - \pe(\tun_L)) \tun_K \\
&= (\Delta t)^2 \sum\limits_{\sigma \in \eint} \frac{\ms}{\ds} (\pe(\tun_K) - \pe(\tun_L)) (\tun_K - \tun_L) \geq 0.
\end{align*}
By Young's inequality combined with the Lipschitz continuity of $\beta$ we obtain
\begin{align*}
	\Delta t \sum\limits_{K \in \Tau} m_K \beta(\tunm_K) \tun_K \leq \frac{\Delta t L_{\beta}^2}{2}  \sum\limits_{K \in \Tau} m_K |\tunm_K|^2 + \frac{\Delta t}{2} \sum\limits_{K \in \Tau} m_K |\tun_K|^2.
\end{align*}
Since $\pe(0)=0$ and thanks to the monotonicity of $\pe$ and we have $\pe(\tun_K)\tun_K\geq 0$. Using the identity $(a-b)a=\frac{1}{2}(a^2-b^2-|a-b|^2)$, and taking expectation, from \eqref{260716_01} it follows that
\begin{align*}
&\frac{1}{2}	\sum\limits_{K \in \Tau} m_K\erww{|\tun_K|^2-|\tunm_K|^2-|\tun_K-\tunm_K|^2}+\erww{\sum\limits_{\sigma \in \eint} \frac{\ms}{\ds} (\tun_K - \tun_L)^2}\\
&\leq 	\sum\limits_{K \in \Tau} m_K\erww{g(\tunm_K) \tun_K \Delta_n W}+\Delta t\sum\limits_{K \in \Tau} m_K\erww{f_K^{n-1} \tun_K}\\
&+\frac{\Delta t L_{\beta}^2}{2}  \sum\limits_{K \in \Tau} m_K \erww{|\tunm_K|^2} + \frac{\Delta t}{2} \sum\limits_{K \in \Tau} m_K \erww{|\tun_K|^2}.
\end{align*}
Since $\tunm_K$ and $g(\tunm_K)$ are $\mathcal{F}_{t_{n-1}}$ measurable and $\Delta_n W$ is independent of $\mathcal{F}_{t_{n-1}}$, we have 
\begin{align*}
\erww{g(\tunm_K) \tunm_K \Delta_n W}= \erww{g(\tunm_K) \tunm_K} \erww{\Delta_n W} = 0.
\end{align*}
Moreover, using the Lipschitz continuity of $g$ and Young's inequality we arrive at
\begin{align*}
&\erww{g(\tunm_K) \tun_K \Delta_n W}=\erww{g(\tunm_K) (\tun_K-\tunm_K) \Delta_n W}\\
&\leq \Delta t L_g^2\erww{|\tunm_K|^2}+\frac{1}{4}\erww{|\tun_K-\tunm_K|^2}.
	\end{align*}
Now we may proceed as in the proof of \cite[Proposition 4.2]{BSVZ25} to close our proof.
\end{proof}

The upcoming theorem provides an error estimate between the solution of the TPFA-scheme and the solution of the splitting method given the same initial datum $u^0$:

\begin{thm}\label{Theorem 020626_01}
Let $u^0 \in \Rtd$ be given. For $N\in\mathbb{N}$, $\eps >0$, and an admissible mesh $\Tau$ in the sense of Definition \ref{defmesh} with mesh size $h>0$, let $\tun$ be the solution of the splitting method given in Definition \ref{splittingmethod}, and $u^n$ be the solution of \eqref{eq_011025_01}, both with respect to the initial datum $u^0$. Then, for $N\in\mathbb{N}$ large enough, there exists a constant $K_2>0$ only depending on $L_g$, $L_{\beta}$, $T$ and $u_0$ such that
\begin{align*}
\sup\limits_{n \in \{1,...,N\}} \erww{\Vert \tun - \un \Vert_{L^2(\Lambda)}^2} \leq K_2 \Delta t \left(1 + \frac{\Delta t}{\eps^2} \right),
\end{align*}
where for $v=(v_K)_{K \in \Tau}$ we identify $ v$ with the simple function $\sum_{K \in \Tau} v_K \mathds{1}_K\in L^2(\Lambda)$.
\end{thm}

\begin{proof}
From \eqref{eq_011025_01} it follows that
\begin{align}\label{eq 290526_02}
\begin{aligned}
&M(\un - \unm) + \Delta t A \un + \Delta t M\pe(\un) \\
&= Mg(\unm)\Delta_n W +\Delta t M f^{n-1} + \Delta t M \beta(\un).
\end{aligned}
\end{align}
Setting $e^n:= \tun - \un$ and subtracting \eqref{eq 290526_02} from \eqref{eq 290526_01}, we obtain
\begin{align*}
&M(e^n - e^{n-1}) + \Delta t A e^n + \Delta t M (\pe(\tun) - \pe(\un)) + (\Delta t)^2 A \pe(\tun) \\
&= M(g(\tunm) - g(\unm))\Delta_n W + \Delta t M (\beta(\tunm) - \beta(\un)).
\end{align*}
Taking the euclidian scalar product with $e^n$ and using the notation $e^n = (e_K^n)_{K \in \Tau}$ yields
\begin{align}\label{eq 290526_03}
\begin{aligned}
&\frac{1}{2} \sum\limits_{K \in \Tau} m_K \left( |e_K^n|^2 - |e_K^{n-1}|^2 + |e_K^n - e_K^{n-1}|^2\right) + \Delta t A e^n \cdot e^n \\
&+ \Delta t \sum\limits_{K \in \Tau} m_K (\pe(\tun_K) - \pe(\un_K))(\tun_K - \un_K) + (\Delta t)^2 A \pe(\tun) \cdot e^n \\
&= \sum\limits_{K \in \Tau} m_K(g(\tunm_K) - g(\unm_K))e_K^n \Delta_n W + \Delta t \sum\limits_{K \in \Tau} m_K (\beta(\tunm) - \beta(\tun))e_K^n \\
&+ \Delta t \sum\limits_{K \in \Tau} m_K (\beta(\tun) - \beta(\un))e_K^n \\
&\Longleftrightarrow I + II + III + IV = V + VI + VII,
\end{aligned}
\end{align}
where
\begin{align}\label{eq 020626_01}
\begin{aligned}
I &= \frac{1}{2} \left( \Vert e^n\Vert_{L^2(\Lambda)}^2 - \Vert e^{n-1}\Vert_{L^2(\Lambda)}^2 + \Vert e^n - e^{n-1}\Vert_{L^2(\Lambda)}^2\right), \\
II&= \Delta t A e^n \cdot e^n = \Delta t \sum\limits_{K \in \Tau} \sum\limits_{\sigma \in \ekint} \frac{\ms}{\ds} (e_K^n - e_L^n) e_K^n \\
&= \Delta t \sum\limits_{\sigma \in \eint} \frac{\ms}{\ds} (e_K^n - e_L^n)^2 = \Delta t |e^n|_{1,h}^2, \\
III &= \Delta t \sum\limits_{K \in \Tau} m_K (\pe(\tun_K) - \pe(\un_K))(\tun_K - \un_K) \geq 0,\\
IV &= (\Delta t)^2 A \pe(\tun) \cdot e^n = (\Delta t)^2 \sum\limits_{K \in \Tau} \sum\limits_{\sigma \in \ekint} \frac{\ms}{\ds} (\pe(\tun_K) - \pe(\tun_L))e_K^n, \\
V &= \sum\limits_{K \in \Tau} m_K(g(\tunm_K) - g(\unm_K))e_K^n \Delta_n W, \\
VI &= \Delta t \sum\limits_{K \in \Tau} m_K (\beta(\tunm_K) - \beta(\tun_K))e_K^n, \\
VII &= \Delta t\sum\limits_{K \in \Tau} m_K (\beta(\tun_K) - \beta(\un_K))e_K^n.
\end{aligned}
\end{align}
Now, thanks to H\"{o}lder and Young's inequality, using the Lipschitz continuity of $\pe$, we have
\begin{align}\label{eq 020626_02}
\begin{aligned}
|IV| &= \left|(\Delta t)^2 A \pe(\tun) \cdot e^n = (\Delta t)^2 \sum\limits_{K \in \Tau} \sum\limits_{\sigma \in \ekint} \frac{\ms}{\ds} (\pe(\tun_K) - \pe(\tun_L))e_K^n \right|\\
&= \left|(\Delta t)^2 \sum\limits_{\sigma \in \eint} \frac{\ms}{\ds} (\pe(\tun_K) - \pe(\tun_L))(e_K^n - e_L^n) \right|\\
&\leq (\Delta t)^2 \left(\sum\limits_{\sigma \in \eint} \frac{\ms}{\ds} (\pe(\tun_K) - \pe(\tun_L))^2\right)^{1/2} \left(\sum\limits_{\sigma \in \eint} \frac{\ms}{\ds} (e_K^n - e_L^n)^2\right)^{1/2} \\
&\leq (\Delta t)^2 \left( \Delta t \sum\limits_{\sigma \in \eint} \frac{\ms}{\ds} (\pe(\tun_K) - \pe(\tun_L))^2 + \frac{1}{4\Delta t} \sum\limits_{\sigma \in \eint} \frac{\ms}{\ds} (e_K^n - e_L^n)^2 \right)\\
&= (\Delta t)^3 \sum\limits_{\sigma \in \eint} \frac{\ms}{\ds} (\pe(\tun_K) - \pe(\tun_L))^2 +\frac{\Delta t}{4} \sum\limits_{\sigma \in \eint} \frac{\ms}{\ds} (e_K^n - e_L^n)^2\\
&\leq \frac{(\Delta t)^3}{\eps^2} \sum\limits_{\sigma \in \eint} \frac{\ms}{\ds} (\tun_K - \tun_L)^2 + \frac{\Delta t}{4} \sum\limits_{\sigma \in \eint} \frac{\ms}{\ds} (e_K^n - e_L^n)^2\\
&= \frac{(\Delta t)^3}{\eps^2} |\tun|_{1,h}^2 + \frac{\Delta t}{4} |e^n|_{1,h}^2.
\end{aligned}
\end{align}
Since $\unm_K$, $\tunm_K$, and $g(\tunm_K)$ are $\mathcal{F}_{t_{n-1}}$ measurable and $\Delta_n W$ is independent of $\mathcal{F}_{t_{n-1}}$, we have 
\begin{align*}
&\erww{(g(\tunm_K) - g(\unm_K))e_K^{n-1}\Delta_n W} \\
=&\erww{(g(\tunm_K) - g(\unm_K))e_K^{n-1}}\erww{\Delta_n W} = 0,
\end{align*}
we obtain by Young's inequality, using the Lipschitz continuity of $g$,
\begin{align}\label{eq 020626_03}
\begin{aligned}
\erww{V} &= \erww{\sum\limits_{K \in \Tau} m_K(g(\tunm_K) - g(\unm_K))e_K^n \Delta_n W } \\
&= \erww{\sum\limits_{K \in \Tau} m_K(g(\tunm_K) - g(\unm_K))(e_K^n - e_K^{n-1}) \Delta_n W } \\
&\leq \erww{\sum\limits_{K \in \Tau} m_K\left( (g(\tunm_K) - g(\unm_K))^2(\Delta_n W)^2+\frac{1}{4}|e_K^n - e_K^{n-1}|^2\right)}\\
&\leq \Delta t \erww{\Vert g(\tunm) - g(\unm)\Vert_{L^2(\Lambda)}^2} + \frac{1}{4} \erww{\Vert e^n - e^{n-1} \Vert_{L^2(\Lambda)}^2} \\
&\leq \Delta t L_g^2 \erww{\Vert e^{n-1} \Vert_{L^2(\Lambda)}^2} + \frac{1}{4} \erww{\Vert e^n - e^{n-1} \Vert_{L^2(\Lambda)}^2}.
\end{aligned}
\end{align}
Moreover, Young's inequality and the Lipschitz continuity of $\beta$ yield
\begin{align}\label{eq 160626_01}
\begin{aligned}
|VI| &\leq \frac{\Delta t L_{\beta}^2}{4} \sum\limits_{K \in \Tau} m_K |\tunm_K - \tun_K|^2 + \Delta t \sum\limits_{K \in \Tau} m_K |e_K^n|^2 \\
&=\frac{\Delta t L_{\beta}^2}{4} \Vert \tun - \tunm \Vert_{L^2(\Lambda)}^2 + \Delta t \Vert e^n \Vert_{L^2(\Lambda)}^2
\end{aligned}
\end{align}
and, using again the Lipschitz continuity of $\beta$ it follows that
\begin{align}\label{eq 160626_02}
\begin{aligned}
|VII| \leq \Delta t L_{\beta} \sum\limits_{K \in \Tau} m_K |\tun_K - \un_K|^2 = \Delta t L_{\beta} \Vert e^n \Vert^2_{L^2(\Lambda)}.
\end{aligned}
\end{align}
Using \eqref{eq 020626_01} to \eqref{eq 160626_02} and taking expectation in \eqref{eq 290526_03}  yields
\begin{align}\label{eq 020626_04}
\begin{aligned}
&\frac{1}{2} \erww{\Vert e^n \Vert_{L^2(\Lambda)}^2} - \frac{1}{2}\erww{\Vert e^{n-1} \Vert_{L^2(\Lambda)}^2} + \frac{1}{4}\erww{\Vert e^n - e^{n-1}\Vert_{L^2(\Lambda)}^2} + \frac{3 \Delta t}{4} \erww{|e^n|_{1,h}^2} \\
&\leq \Delta t L_g^2 \erww{\Vert e^{n-1} \Vert_{L^2(\Lambda)}^2} + \frac{(\Delta t)^2}{\eps^2}\Delta t \erww{|\tilde{u}^n|_{1,h}^2} + \Delta t(1+L_{\beta}) \erww{\Vert e^n \Vert_{L^2(\Lambda)}^2} \\
&+ \frac{\Delta t L_{\beta}^2}{4} \erww{\Vert \tun - \tunm \Vert_{L^2(\Lambda)}^2}.\\
\end{aligned}
\end{align}
Now, exchanging the variables $n$ and $m$ in \eqref{eq 020626_04}, summing $m$ from $1$ to $n$ and using $e^0 = 0$ we arrive at
\begin{align*}
&\frac{1}{2} \erww{\Vert e^n\Vert_{L^2(\Lambda)}^2} + \frac{1}{4} \sum\limits_{m=1}^n \erww{\Vert e^m - e^{m-1} \Vert_{L^2(\Lambda)}^2} + \frac{3 \Delta t}{4} \sum\limits_{m=1}^n \erww{|e^m|_{1,h}^2} \\
&\leq \Delta t L_g^2 \sum\limits_{m=1}^n \erww{\Vert e^{m-1} \Vert_{L^2(\Lambda)}^2} +\Delta t\sum_{m=1}^n\frac{(\Delta t)^2}{\eps^2}\erww{|\tilde{u}^m|_{1,h}^2}\\
&+\Delta t(1+L_{\beta})\sum_{m=1}^n \erww{\Vert e^m \Vert_{L^2(\Lambda)}^2}+\frac{\Delta t L_{\beta}^2}{4} \sum_{m=1}^n\erww{\Vert \tilde{u}^m - \tilde{u}^{m-1} \Vert_{L^2(\Lambda)}^2}.
\end{align*}
For $n=1,\ldots,N$ and $N\in\mathbb{N}$ large enough, using Lemma \ref{Lemma 290526_01}, shifting indices, and setting $C_{\beta}:=1+L_{\beta}>0$ we now obtain
\begin{align*}
&\frac{1}{2} \erww{\Vert e^n\Vert_{L^2(\Lambda)}^2} + \frac{1}{4} \sum\limits_{m=1}^n \erww{\Vert e^m - e^{m-1} \Vert_{L^2(\Lambda)}^2} + \frac{3 \Delta t}{4} \sum\limits_{m=1}^n \erww{|e^m|_{1,h}^2} \\
&\leq \Delta t (L_g^2+C_{\beta})\sum\limits_{m=1}^n \erww{\Vert e^{m-1} \Vert_{L^2(\Lambda)}^2} +\Delta tC_{\beta}\erww{\Vert e^{n} \Vert_{L^2(\Lambda)}^2}+\Delta t K_1\left(\frac{\Delta t}{\eps^2} + \frac{L_{\beta}^2}{4} \right)
\end{align*}
and therefore
\begin{align*}
&\left(\frac{1}{2}- C_{\beta}\Delta t \right) \erww{\Vert e^n\Vert_{L^2(\Lambda)}^2} + \frac{1}{4} \sum\limits_{m=1}^n \erww{\Vert e^m - e^{m-1} \Vert_{L^2(\Lambda)}^2} + \frac{3 \Delta t}{4} \sum\limits_{m=1}^n \erww{|e^m|_{1,h}^2} \\
&\leq \Delta t (L_g^2+ C_{\beta}) \sum\limits_{m=1}^n \erww{\Vert e^{m-1} \Vert_{L^2(\Lambda)}^2} + \left(\frac{\Delta t}{\eps^2} + \frac{L_{\beta}^2}{4} \right) \Delta t K_1,
\end{align*}
where $K_1>0$ only depends on $L_{\beta}$, $L_g$ and $T$. For $N \geq 4C_{\beta}T$ we have $\frac{1}{2} - C_{\beta}\Delta t \geq \frac{1}{4}$ and therefore applying the discrete Gronwall lemma (see, e.g., \cite[Lemma 1]{S69}) we obtain a constant $K_2>0$ only depending on $K_1$, $L_{\beta}$, $L_g$ and $T$ such that
\begin{align*}
\erww{\Vert \tun - \un\Vert_{L^2(\Lambda)}^2}= \erww{\Vert e^n\Vert_{L^2(\Lambda)}^2} \leq K_2 \Delta t \left(1 + \frac{\Delta t}{\eps^2} \right),
\end{align*}
for $n=1,\ldots,N$ and $N\in\mathbb{N}$ large enough.
\end{proof}

\begin{cor}\label{Corollary 160626_01}
Let $u^0 \in \Rtd$ be given. For $N\in\mathbb{N}$, $\eps >0$, and an admissible mesh $\Tau$ in the sense of Definition \ref{defmesh} with mesh size $h>0$, let $\tun$ be the solution of the splitting method given in Definition \ref{splittingmethod}, and $u^n$ be the solution of \eqref{eq_011025_01}, both with respect to the initial datum $u^0$. Then
\begin{align*}
\sup\limits_{n \in \{1,...,N\}} \erww{\Vert \pe(\tun) - \pe(\un) \Vert_{L^2(\Lambda)}^2} \leq K_2 \frac{\Delta t}{\eps^2} \left(1 + \frac{\Delta t}{\eps^2} \right),
\end{align*}
for $N\in\mathbb{N}$ large enough.
\end{cor}

\begin{proof}
The Lipschitz continuity of $\pe$ and the result of Theorem \ref{Theorem 020626_01} yield that for $N \in \N$ as in Theorem \ref{Theorem 020626_01} we have
\begin{align*}
&\sup\limits_{n \in \{1,...,N\}} \erww{\Vert \pe(\tun) - \pe(\un) \Vert_{L^2(\Lambda)}^2} \leq \frac{1}{\eps^2} \sup\limits_{n \in \{1,...,N\}} \erww{\Vert \tun - \un \Vert_{L^2(\Lambda)}^2} \\
&\leq K_2 \frac{\Delta t}{\eps^2} \left(1 + \frac{\Delta t}{\eps^2} \right).
\end{align*}
\end{proof}

\begin{rem}
The choice of $N\in\mathbb{N}$ in Lemma \ref{Lemma 290526_01}, Theorem \ref{Theorem 020626_01}, and Corollary \ref{Corollary 160626_01} only depends on $L_{\beta}$ and $T$.
\end{rem}

\begin{cor}\label{260710_rem1}
For $\theta, C>0$ and $0< \eps \leq 1$, let us assume that $N \in \N$ is large enough such that Theorem \ref{Theorem 020626_01} and Corollary \ref{Corollary 160626_01} are satisfied and $\Delta t \leq C \eps^{2+\theta}$. Then, we obtain
\begin{align}\label{260710_01}
\sup\limits_{n \in \{1,...,N\}}\erww{\Vert \tun - \un\Vert_{L^2(\Lambda)}^2}\leq (1 + C\eps^{\theta})K_2 \Delta t  \leq (1+C)K_2 \Delta t
\end{align}
and
\begin{align}\label{260716_02}
&\sup\limits_{n \in \{1,...,N\}} \erww{\Vert \pe(\tun) - \pe(\un) \Vert_{L^2(\Lambda)}^2} \leq C\eps^{\theta}(1 + C\eps^{\theta})K_2   \leq C(1+C)K_2 \eps^{\theta}.
\end{align}
\end{cor}
Let $(u,\psi)$ be the unique solution of Problem \eqref{equation} in the sense of Definition \ref{solution} and let $(\Tau_m)_{m \in \N}$, $(h_m)_{m \in \N}$, $(N_m)_{m \in \N}$, and $(\eps_m)_{m \in \N}$ be chosen as in Theorem \ref{Theorem 020626_02}. Combining Corollary \ref{260710_rem1} with Theorem \ref{Theorem 020626_02}, we may conclude that, assuming $\Delta t \in \mathcal{O}(\eps^{2+\theta})$, the sequence $(\tilde{u}_{h_m, N_m, \eps_m})_{m \in \N}$ given by the splitting scheme converges to $u$ strongly and the sequence $\psi_{\eps_m}((\tilde{u}_{h_m, N_m, \eps_m}))_{m \in \N}$ converges to $\psi$ weakly in $L^2(\Omega; L^2(0,T; L^2(\Lambda)))$.

\section{Miscellaneous properties}\label{properties}
\subsection{Properties of the solution matrix}
The following lemma states some properties of the matrix $(M + \Delta t A)^{-1} M$ where $M$, $A$ are given in \eqref{260709_01} and \eqref{260709_02}, respectively. If $Z$ is a vector or a matrix, we denote $Z \geq 0$ if and only if all components of $Z$ are non-negative. Similarly, we denote $Z \leq 0$ if and only if $-Z \geq 0$.
\begin{lem}\label{Lemma 151025_01}
Let $x \in \Rtd$ with $x \geq 0$. Then $(M+ \Delta t A)^{-1}Mx \geq 0$.
\end{lem}
\begin{proof}
Since $A$ is symmetric and positive semi-definite with non-positive off-diagonal components, $M+ \Delta t A$ is symmetric and positive definite with non-positive off-diagonal components. Hence, $M+ \Delta t A$ is a Stieltjes matrix and therefore $(M+ \Delta t A)^{-1} \geq 0$ (see \cite[Theorem 7.3]{Y03}). Especially, if $x \in \Rtd$ with $x \geq 0$, then $(M + \Delta t A)^{-1} x \geq 0$.
\end{proof}

\subsection{Analysis of the projection operator}\label{projection}
In accordance with Corollary \ref{Corollary 160626_01} and \ref{260710_rem1} we assume that $\eps=C(\Delta t)^{\frac{1}{2+\theta}}$ for free parameters $\theta,C>0$. For $C$ and $\Delta t$ fixed, $\eps>0$ decreases as $\theta>0$ decreases and, from the theoretical point of view, the projection effect of the Yosida approximation improves with decreasing $\eps>0$. However, since $\eps$ is coupled to the time step in the splitting scheme, the projection operator given in \eqref{eq_011025_07} expressed in dependence of $C$, $\theta$ and $\Delta t$ then reads as: 
\begin{align*}
(I+\Delta t\pe)^{-1}(y)=\begin{cases}
\left(1+\frac{(\Delta t)^{\frac{1+\theta}{2+\theta}}}{C}\right)^{-1}y, &y<0 \\
y, &0\leq y\leq 1\\
\left(1+\frac{(\Delta t)^{\frac{1+\theta}{2+\theta}}}{C}\right)^{-1}y+\frac{(\Delta t)^{\frac{1+\theta}{2+\theta}}}{C+(\Delta t)^{\frac{1+\theta}{2+\theta}}}, &y>1.
\end{cases} 
\end{align*}
For the strength of the projection at a single time iteration step, the parameter $C>0$ may be helpful. The smaller $C>0$ is chosen, the stronger is the projection effect in one time step. In Section \ref{num}, we will provide numerical experiments which will illustrate the influence of the parameters $\theta>0$ and $C>0$ on the projection operator.
\subsection{Properties for \texorpdfstring{$\beta=0$ and $f=0$}{beta=0 and f=0}}\label{noforce}
In this subsection, we study problem \eqref{eqeps} in the special case without external forcing, i.e., $\beta=0$ and $f=0$. First we recall that, in this case, its solution $u_\eps$ will stay between $0$ and $1$ if $0\leq u_0\leq 1$ a.s. in $\Omega\times \Lambda$ thanks to the following maximum principle. Its proof is very similar to the proof of \cite[Th. 7.1, p. 154]{Pardoux}. For the convenience of the reader, we will sketch the main arguments.
\begin{lem}\label{260826_lem1}
Let $h:\mathbb{R}\rightarrow\mathbb{R}$ be Lipschitz continuous with $\operatorname{supp} h\subset[0,1]$ and $\xi:\mathbb{R}\rightarrow\mathbb{R}$ be non-decreasing with $\xi(0)=0$. Let $z$ be the unique solution to the problem
\begin{align}\label{260826_01}
\begin{aligned}
dz+(\xi(z)-\Delta z)\,dt &=h(z)\,dW(t), &&\text{ in }\Omega\times(0,T)\times\Lambda;\\
z(0,\cdot)&=z_0, &&\text{ in } \Omega\times\Lambda;\\
\nabla z\cdot \mathbf{n}&=0, &&\text{ on }\Omega\times(0,T)\times\partial\Lambda;
\end{aligned}
\end{align}
for a $\mathcal{F}_0$-measurable initial datum $z_0\in L^2(\Omega\times\Lambda)$ such that $z_0\geq 0$ a.e. in $\Omega\times \Lambda$. Then, $z(t,x)\geq 0$ a.e. in $\Lambda$, for all $t\in [0,T]$, a.s. in $\Omega$.
\end{lem}
\begin{proof}
For $\delta>0$ we define a smooth approximation of $\eta:\mathbb{R}\rightarrow\mathbb{R}$, $\eta(r)=|r^-|^2$ by $\eta_{\delta}:\mathbb{R}\rightarrow\mathbb{R}$,
\[\eta_{\delta}(r)=\begin{cases}
r^2-\frac{\delta^2}{6} & \text{if }r<-\delta\\
-\frac{r^4}{2\delta^2}-\frac{4r^3}{3\delta} & \text{if }-\delta\leq r\leq 0\\
0 & \text{if }r>0.\end{cases}\] 
It\^{o}'s formula yields
\[I_1+I_2+I_3+I_4=I_5 + I_6\]
for all $t\in [0,T]$, a.s. in $\Omega$, where
\begin{align*}
I_1&=\int_{\Lambda} \eta_{\delta}(z(t,x))\,dx, \\
I_2&= \int_{\Lambda} \eta_{\delta}(z_0(x)) \,dx, \\
I_3&=\int_0^t\int_{\Lambda} |\nabla z(s,x)|^2\eta_{\delta}''(z(s,x))\,dx\,ds, \\
I_4&=\int_0^t\int_{\Lambda} \xi(z(s,x))\eta_{\delta}'(z(s,x))\,dx\,ds, \\
I_5&=\int_0^t\int_{\Lambda} h(z(s,x))\eta_{\delta}'(z(s,x))\,dx\,dW(s),\\
I_6&=\int_0^t\int_{\Lambda} h^2(z(s,x))\eta_{\delta}''(z(s,x))\,dx\,ds.
\end{align*}
Since $\eta_{\delta}(r)=0$ for $r \geq 0$, we have $\eta_{\delta}(z_0) = 0$ a.e. in ${\Lambda}$ and therefore $I_2=0$.
Recalling that $\eta_{\delta}''(r)\geq 0$ for all $r\in\mathbb{R}$ it follows that
\begin{align*}
I_3&=\int_0^t\int_{\Lambda} |\nabla z(s,x)|^2\eta_{\delta}''(z(s,x))\,dx\,ds\geq 0.
\end{align*}
Since $\eta_{\delta}'(r)=0$ for $r > 0$,  $\eta_{\delta}'(r)\leq 0$ for $r \leq 0$, and $\xi(r)\leq 0$ for $r\leq 0$, it follows that 
\begin{align*}
I_4&=\int_0^t\int_{\Lambda} \xi(z(s,x))\eta_{\delta}'(z(s,x))\,dx\,ds\geq 0.
\end{align*}
Using that $\operatorname{supp}h\subset[0,1]$ and $\eta_{\delta}'(r)=0$ for $r \geq 0$, we find that
\begin{align*}
I_5&=\int_0^t\int_{\Lambda} h(z(s,x))\eta_{\delta}'(z(s,x))\,dx\,dW(s)=0,\\
I_6&=\int_0^t\int_{\Lambda} h^2(z(s,x))\eta_{\delta}'(z(s,x))\,dx\,ds=0.
\end{align*}
Discarding $I_3$ and $I_4$, then passing to the limit with $\delta\rightarrow 0^+$ we obtain
\begin{align*}
\int_{\Lambda} |z(t,x)^-|^2\,dx\leq 0
\end{align*}
for all $t\in [0,T]$, a.s. in $\Omega$, hence the result.
\end{proof}
With the help of this maximum principle, we may establish the desired result:
\begin{lem}\label{260826_lem2}
For fixed $\eps>0$ let $u_\eps$ be the solution to \eqref{eqeps} with $f=0$ and $\beta=0$ for a $\mathcal{F}_0$-measurable initial datum $u_0\in L^2(\Omega\times\Lambda)$ such that $0\leq u_0\leq 1$ a.e. in $\Omega\times \Lambda$. Then, $0\leq u_{\eps}(t,x)\leq 1$ a.e. in $\Lambda$, for all $t\in [0,T]$, a.s. in $\Omega$.
\end{lem}
\begin{proof}
Recalling that $\pe$ is monotone with $\pe(0)=0$, from Lemma \ref{260826_lem1} we immediately obtain $0\leq u_\eps(t,x)$ a.e. in $\Lambda$, for all $t\in [0,T]$, a.s. in $\Omega$. Now, we set $v_\eps:=1-u_\eps$. Consequently,
\begin{align*}
v_{\eps} - v_{\eps}(0) = (1 - u_{\eps}) - (1 - u_{\eps}(0)) = -(u_{\eps} - u_0).
\end{align*}
Moreover, we define $\widetilde{g}(r):=-g(1-r)$ for all $r \in \R$. Then, $\widetilde{g}(v_{\eps})=-g(u_{\eps})$. Since $\pe(r)=-\pe(1-r)$ for all $r \in \R$ we obtain $\pe(v_{\eps})=-\pe(u_{\eps})$.
Furthermore, $\Delta v_{\eps} = - \Delta u_{\eps}$ and it is easy to see that
\begin{align*}
v_{\eps}(t)-v_{\eps}(0)+\int_0^t(\pe(v_\eps)-\Delta v_\eps)\,ds=\int_0^t \widetilde{g}(v_\eps)\,dW(s)
\end{align*}
for all $t\in [0,T]$, a.s. in $\Omega$, where $\Delta$ denotes the Laplace operator on $H^1(\Lambda)$ associated with the weak Neumann boundary condition.
Since $v_\eps(0)\geq 0$ a.e. in $\Omega\times \Lambda$, $\widetilde{g}:\mathbb{R}\rightarrow\mathbb{R}$ is Lipschitz continuous with $\operatorname{supp}\widetilde{g}\subset [0,1]$, from Lemma \ref{260826_lem1} it follows that $0\leq v_\eps(t,x)$ a.e. in $\Lambda$, for all $t\in [0,T]$, a.s. in $\Omega$, hence $u_\eps(t,x)\leq 1$ a.e. in $\Lambda$, for all $t\in [0,T]$, a.s. in $\Omega$.
\end{proof}
\begin{rem}\label{260826_rem1}
As a consequence, in the case of no external forcing, we have $\pe(u_\eps)=0$ a.s in $\Omega\times (0,T)\times \Lambda$. Therefore, $u_{\eps}$ solves the stochastic heat equation with Lipschitz noise and no constraint. By pathwise uniqueness, we have $u_{\eps}=u$ for all $\eps>0$. Moreover, $(u,0)$ is the unique solution to \eqref{equation}. Since there is no need to put a projection operator $\pe$ into the equation from the analytic point of view, the TPFA scheme \eqref{eq_011025_01} and the splitting method in Definition \ref{splittingmethod} coincide in the case of no external forcing and are identical to the TPFA scheme for a stochastic heat equation in \cite{BNSZ22}.
\end{rem}
We now investigate whether, and in what sense, the structural properties of the solutions to \eqref{eqeps} in the case of no external forcing are preserved by the TPFA scheme \eqref{eq_011025_01} and by the splitting method defined in Definition \ref{splittingmethod}. Taking Remark \ref{260826_rem1} into consideration, we will set $\pe=0$ and not distinguish between the two methods in what follows.
\subsubsection{Sign property}
In the following, we will denote by $\mathbf{1} \in \Rtd$ the vector whose components are all equal to $1$.

\begin{prop}\label{Proposition 151025_01}
We assume that $f=0$ and $\beta=0$. For $n=1,\ldots,N$, let $\unm \in \Rtd$ such that $\unm \leq 0$ or $\unm - \mathbf{1} \geq 0$, respectively. Moreover, let $\un \in \Rtd$ be the solution of the subsequent step of the TPFA-scheme \eqref{eq_011025_01}. Then $\un \leq 0$ or $\un - \mathbf{1} \geq 0$, respectively.
\end{prop}
\begin{proof}
 Let $\unm \in \Rtd$ with $\unm \leq 0$ or $\unm - \mathbf{1} \geq 0$, respectively. Since $\operatorname{supp} g \subset [0,1]$ we have $g(\unm) = 0 \in \Rtd$ and therefore we obtain
\begin{align*}
\un &=(M + \Delta t A)^{-1}M (\unm + g(\unm) \Delta_n W) \\
&=(M + \Delta t A)^{-1}M(\unm).
\end{align*}
Now, assuming $\unm \leq 0$, Lemma \ref{Lemma 151025_01} yields $(M + \Delta t A)^{-1}M(\unm) \leq 0$ and therefore $\un \leq 0$.
Since $A \mathbf{1} = 0$, we have $\mathbf{1} = (M + \Delta t A)^{-1} M (\mathbf{1})$. Consequently, for $\unm - \mathbf{1} \geq 0$, from Lemma \ref{Lemma 151025_01} we get
\begin{align*}
\un-\mathbf{1} = (M + \Delta t A )^{-1}M(\unm) - \mathbf{1} = (M + \Delta t A )^{-1}M(\unm - \mathbf{1}) \geq 0.
\end{align*}
\end{proof}
\subsubsection{Constant solutions}
In this subsection we show that the TPFA-scheme preserves constant solutions of problem \eqref{eqeps} in the special case without external forcing, i.e., $\beta=0$ and $f=0$. Then it not very hard to see that if the initial value is constant in space a.s. in $\Omega$, the solution of \eqref{eqeps} will be constant in space a.s. in $\Omega$. The next result shows that this property also applies to the TPFA-scheme.
\begin{prop}\label{Theorem 161025_01}
We assume that $f=0$ and $\beta=0$. Let $u^0 \in \Rtd$ be pathwise constant, i.e., for any $\omega \in \Omega$ there exists $c \in [0,1]$ such that $u^0_K(\omega)= c$ for each $K\in \Tau$. Moreover, for $N \in \N$ and $n\in \{1,\ldots,N\}$ let $u^n = (u^n_K)_{K \in \Tau}$ be the solution of the TPFA-scheme \eqref{eq_011025_01} after $n$ time steps. 
Then, for any $\omega \in \Omega$, all components $u_K^n(\omega)$ of $(u^n_K(\omega))_{K \in \Tau}$ are identical. 
\end{prop}
\begin{proof}
We argue pathwise, but drop the dependence on $\omega$ in the notation. By induction, we only need to proof the assertion for $n=1$. The vector $u^1=(u^1_K)_{K \in \Tau}$ is given by the formula
\begin{align*}
u^1 &= (M + \Delta t A)^{-1} M (u^0 + g(u^0) \Delta_1 W).
\end{align*}
Since
\begin{align*}
(u^0 + g(u^0) \Delta_1 W)_K &= u^0_K + g(u^0_K) \Delta_1 W =u^0_L + g(u^0_L) \Delta_1 W \\
&= (u^0 + g(u^0) \Delta_1 W)_L
\end{align*}
for all $K,L\in \Tau$ and recalling $\mathbf{1} = (M + \Delta t A)^{-1} M (\mathbf{1})$,
we get
\begin{align*}
u_K^1 = u^0_K + g(u^0_K) \Delta_1 W =u^0_L + g(u^0_L) \Delta_1 W = u_L^1
\end{align*}
for all $K,L\in \Tau$.
\end{proof}
Let $c\in [0,1]$ be such that $g(c)=0$. If $u$ is a solution to \eqref{eqeps} in the special case $f=0$ and $\beta=0$ with initial value $u_0(\omega,x)=c$ a.s. in $\Omega\times\Lambda$, then $u(\omega,t,x)=c$ a.s. in $\Omega\times (0,T)\times \Lambda$. The same result applies to the TPFA-scheme:
\begin{cor}\label{Theorem 161025_02}
We assume that $f=0$ and $\beta=0$ and fix $c\in\mathbb{R}$ such that $g(c)=0$. Let $u^0 \in \Rtd$ be such that $u^0_K= c$ for each $K\in \Tau$ a.s. in $\Omega$. Moreover, for $N \in \N$ and $n\in \{1,\ldots,N\}$ let $u^n = (u^n_K)_{K \in \Tau}$ be the solution of the TPFA-scheme \eqref{eq_011025_01} after $n$ time steps. 
Then $u_K^n=c$ for all $K \in \Tau$ a.s. in $\Omega$.
\end{cor}

\section{Numerical experiments}\label{num}
\subsection{General setting}\label{gs}
All upcoming experiments were performed on a single workstation in Python. In our setting, $T=1$ and $\Lambda =(-1,1)^2$. We consider a uniform Cartesian mesh $\Tau$ on $\overline{\Lambda}$ consisting of square cells of equal size which are our control volumes $K\in\Tau$. The parameter $L$ will denote the number of squares in each spatial direction, i.e., the side length of each cell $K \in \Tau$ is equal to $2 / L$ and therefore we have $h = \sqrt{8}/L$. The total number of subsquares used for the spatial discretization is $d_h=L^2$ and the 2-d Lebesgue measure of each subsquare $K\in\Tau$ is $m_K=4/L^2$. Our initial value
\begin{align*}
u_0(x,y) = \left(\frac{1}{16}x^4 + \frac{1}{4}x^3 - \frac{1}{8}x^2 - \frac{3}{4}x + \frac{9}{16}\right)\left(\frac{3}{32}y^4 - \frac{1}{4}y^3 - \frac{3}{16}y^2 + \frac{3}{4}y +\frac{19}{32}\right)
\end{align*}
is non-symmetric, smooth, such that $0 \leq u_0(x,y) \leq 1$ for all $x,y \in [-1,1]$, and satisfying the homogeneous Neumann boundary condition in the classical sense. We set $u^0 := u_h^0 := (u_K^0)_{K \in \Tau}$ as in Section \ref{Section Introduction}. The diffusion coefficient we use is
\begin{align*}
g(r) = r(1-r)\mathds{1}_{[0,1]}(r), ~r \in \R.
\end{align*}
For convenience, we choose $\beta$ to be linear, more precisely, $\beta(r)=2r$ for $r\in\mathbb{R}$. Moreover, we use the adapted and continuous external force 
\[f(\omega,t,x,y)=(-2t^2+0.5)(x^2+5y+2)+aW(t)\] 
with parameter $a=2$. The function $f$ is discretized as $f^{n-1}_K=f(t_{n-1},x_K,y_K)$ for $n=1,\ldots,N$ where $(x_K,y_K)$ are the coordinates of the center of the cell $K\in\Tau$.
The discretization parameters $\eps$ and $\Delta t$ are coupled as described in Subsection \ref{projection}, i.e., $\eps=C(\Delta t)^{\frac{1}{2+\theta}}$ for free parameters $\theta,C>0$.
\subsection{Numerical study of the projection effect}
In the following we want to illustrate the projection effect of the splitting method given in Definition \ref{splittingmethod}. To this end, the approximate solution of the Neumann problem for the stochastic Allen-Cahn equation $(StAC)$ and the approximate solution of the corresponding Neumann problem for the stochastic heat equation $(SHE)$ are calculated simultaneously along identical sample paths with the same data $(u_0,f,\beta)$ given in Subsection \ref{gs}. For the solution of $(SHE)$ we use the splitting method with $\pe=0$.
\subsubsection{Pathwise properties}\label{path}
Let us compare an approximate solution for $(StAC)$ with an approximate solution for $(SHE)$ using the same data $(u_0,f,\beta)$ specified in Subsection \ref{gs} along a single sample path that is fixed throughout all experiments. In this experiment, we set $L=40$, $N=2^{10}$ and $a=2$. We consider the cases $C=1$ and $\theta\in\{100,10, 0.01\}$. In this way, we are able to compare the strenght of the projection effect on the approximate solution to $(StAC)$ for different parameters $\theta>0$. Then, we consider as well $C=0.1$ and $\theta=10$ to examine how the parameter $C>0$ affects the projection. In our implementation, we first determine the values of $\Delta_n W$ for $n=1,\ldots,N$ using Python's random number generator with a fixed seed. Then we evaluate the approximate values $u^{StAC}$ of $(StAC)$ and the approximate values $u^{SHE}$ of $(SHE)$ simultaneously. In Tables \ref{table theta=100}, \ref{table theta=10} and \ref{table theta=0.01} we show these values for $t=0$, $t=0.5$, $t=1$ (listed in the columns), and every 200th spatial cell (listed in the rows) for different configurations of $\theta>0$ and $C>0$. The column $t=0$ reflects the common initial data for both problems. In Table \ref{table theta=100}, we set $\theta=100$. In this case, the values $u^{StAC}_{0.5}$ and $u^{StAC}_{1}$ are far outside the interval $[0,1]$ and close to $u^{SHE}_{0.5}$, $u^{SHE}_{1}$ respectively, thus the projection effect is very weak. In the intermediate case $\theta=10$, shown in Table \ref{table theta=10}, the projection effect on the approximate solution $u^{StAC}_1$ of $(StAC)$ in $t=1$ is clearly visible. Along the same sample path, the value $u^{SHE}_1$ of the approximate solution to $(SHE)$ is obviously larger than the corresponding value $u^{StAC}_1$ for $(StAC)$. However, the approximate solution of $(StAC)$ still takes values far outside the interval $[0,1]$. In Table \ref{table theta=0.01}, where $\theta=0.01$, the projection effect is already visible at $t=0.5$, i.e., for $u^{StAC}_{0.5}$. At $t=1$, $u^{StAC}_1$ and $u^{SHE}_1$ differ significantly along the same sample path. Moreover, $u^{StAC}_1$ is close to $1$, while $u^{SHE}_1$ is not. In Table \ref{table theta=10 C=0.1} we keep the parameter $\theta=10$ and set $C=0.1$. Comparing the values of $u^{StAC}$ and $u^{SHE}$ in this case, we observe that the (smaller) constant $C=0.1$ improved the behaviour of the projection operator. Since the projection effect becomes stronger as $\theta>0$ decreases, and can be further improved by choosing a small constant $C>0$, our pathwise numerical experiments support and illustrate the theoretical results on the projection operator presented in Subsection \ref{projection}.

\begin{table}[htbp]
\centering
\caption{Comparison of \((StAC)\) and \((SHE)\) for \(\theta = 100\), \(C = 1\).}
\label{table theta=100}
\small
\begin{tabular}{rccccc}
\hline
\textbf{Cell}
& \(u_{0}\)
& \(u^{StAC}_{0.5}\)
& \(u^{SHE}_{0.5}\)
& \(u^{StAC}_{1}\)
& \(u^{SHE}_{1}\)\rule[0pt]{0pt}{2.6ex} \\
\hline
   0 & 0.00092 & 1.67099 & 1.77844 & 4.39496 & 6.59477\\
 200 & 0.07272 & 1.68563 & 1.79808 & 4.42489 & 6.63475\\
 400 & 0.22673 & 1.72359 & 1.84885 & 4.47919 & 6.71337\\
 600 & 0.41842 & 1.78078 & 1.92520 & 4.53251 & 6.80011\\ 
 800 & 0.61208 & 1.84967 & 2.01710 & 4.56920 & 6.87308\\ 
1000 & 0.78076 & 1.92016 & 2.11107 & 4.58182 & 6.91901\\
1200 & 0.90630 & 1.98101 & 2.19210 & 4.57126 & 6.93422\\
1400 & 0.97933 & 2.02162 & 2.24606 & 4.54820 & 6.92669\\
\hline
\end{tabular}
\end{table}

\begin{table}[htbp]
\centering
\caption{Comparison of \((StAC)\) and \((SHE)\) for \(\theta = 10\), \(C = 1\).}
\label{table theta=10}
\small
\begin{tabular}{rccccc}
\hline
\textbf{Cell}
& \(u_{0}\)
& \(u^{StAC}_{0.5}\)
& \(u^{SHE}_{0.5}\)
& \(u^{StAC}_{1}\)
& \(u^{SHE}_{1}\)\rule[0pt]{0pt}{2.6ex} \\
\hline
   0 & 0.00092 & 1.61051 & 1.77844 & 3.53829 & 6.59477\\
 200 & 0.07272 & 1.62256 & 1.79808 & 3.56451 & 6.63475\\
 400 & 0.22673 & 1.65386 & 1.84885 & 3.60999 & 6.71337\\
 600 & 0.41842 & 1.70110 & 1.92520 & 3.65166 & 6.80011\\
 800 & 0.61208 & 1.75808 & 2.01710 & 3.67626 & 6.87308\\
1000 & 0.78076 & 1.81642 & 2.11107 & 3.67838 & 6.91901\\
1200 & 0.90630 & 1.86686 & 2.19210 & 3.66024 & 6.93422\\
1400 & 0.97933 & 1.90057 & 2.24606 & 3.63295 & 6.92669\\
\hline
\end{tabular}
\end{table}

\begin{table}[htbp]
\centering
\caption{Comparison of \((StAC)\) and \((SHE)\) for \(\theta = 0.01\), \(C = 1\).}
\label{table theta=0.01}
\small
\begin{tabular}{rccccc}
\hline
\textbf{Cell}
& \(u_{0}\)
& \(u^{StAC}_{0.5}\)
& \(u^{SHE}_{0.5}\)
& \(u^{StAC}_{1}\)
& \(u^{SHE}_{1}\)\rule[0pt]{0pt}{2.6ex} \\
\hline
   0 & 0.00092 & 1.09294 & 1.77844 & 1.21088 & 6.59477\\
 200 & 0.07272 & 1.09270 & 1.79808 & 1.22054 & 6.63475\\
 400 & 0.22673 & 1.09233 & 1.84885 & 1.23307 & 6.71337\\
 600 & 0.41842 & 1.09212 & 1.92520 & 1.24151 & 6.80011\\
 800 & 0.61208 & 1.09217 & 2.01710 & 1.24403 & 6.87308\\
1000 & 0.78076 & 1.09249 & 2.11107 & 1.24027 & 6.91901\\
1200 & 0.90630 & 1.09301 & 2.19210 & 1.23078 & 6.93422\\
1400 & 0.97933 & 1.09355 & 2.24606 & 1.21798 & 6.92669 \\
\end{tabular}
\end{table}

\begin{table}[htbp]
\centering
\caption{Comparison of \((StAC)\) and \((SHE)\) for \(\theta = 10\), \(C = 0.1\).}
\label{table theta=10 C=0.1}
\small
\begin{tabular}{rccccc}
\hline
\textbf{Cell}
& \(u_{0}\)
& \(u^{StAC}_{0.5}\)
& \(u^{SHE}_{0.5}\)
& \(u^{StAC}_{1}\)
& \(u^{SHE}_{1}\)\rule[0pt]{0pt}{2.6ex} \\
\hline
  0 & 0.00092 & 1.14625 & 1.77844 & 1.30845 & 6.59477\\
 200 & 0.07272 & 1.14611 & 1.79808 & 1.32046 & 6.63475\\
 400 & 0.22673 & 1.14609 & 1.84885 & 1.33699 & 6.71337\\
 600 & 0.41842 & 1.14654 & 1.92520 & 1.34856 & 6.80011\\
 800 & 0.61208 & 1.14753 & 2.01710 & 1.35206 & 6.87308\\
1000 & 0.78076 & 1.14894 & 2.11107 & 1.34684 & 6.91901\\
1200 & 0.90630 & 1.15052 & 2.19210 & 1.33391 & 6.93422\\
1400 & 0.97933 & 1.15186 & 2.24606 & 1.31719 & 6.92669\\
\hline
\end{tabular}
\end{table}
\subsubsection{Study of the projection in \texorpdfstring{$L^2$}{L²}-norm}\label{Study projection}
We keep denoting the approximate solutions of $(StAC)$ by $u^{StAC}$ and the approximate solutions to $(SHE)$ by $u^{SHE}$, respectively. We consider their difference in the square of the norm $L^2(\Omega\times(0,T)\times \Lambda)$, i.e., the discrete counterpart of
\[\erww{\int_0^T\Vert u^{StAC}(t)-u^{SHE}(t)\Vert^2_{L^2(\Lambda)}\,dt},\]
that reads as
\begin{align}\label{260903_01}
E(\theta,C,\mathcal{P},N,L)=\frac{4\Delta t}{\mathcal{P}L^2}\sum_{p=1}^{\mathcal{P}}\sum_{k=1}^N|u^{StAC}_{\theta,C,p,k,h}-u^{SHE}_{p,k,h}|^2_{\mathbb{R}^{d_h}},
\end{align}
where $\mathcal{P}$ is the number of sample paths used in the classical Monte Carlo method and, as usual, $\theta,C>0$ are the parameters of the projection operator, $N\in\mathbb{N}$ is the number of time steps, $L\in\mathbb{N}$ the number of squares in each spatial direction, $h>0$ is the mesh size and $d_h$ the number of control volumes. We want to study the dependence of \eqref{260903_01} on $C,\theta>0$. In the first experiment, we fix the parameters $\mathcal{P}=1000$, $N=2^{10}$, $L=40$ and $C=1$. We compute \eqref{260903_01} for $\theta\in \{0.01, 0.05, 0.1, 0.5, 1, 2, 5, 10, 20, 50, 100 \}$.
\begin{table}[htbp]
\centering
\caption{$E(\theta,1,1000,2^{10},40)$ for variable
$\theta>0$.}
\label{table:theta-var}
\small
\begin{tabular}{rr}
\hline
\multicolumn{1}{c}{$\theta$} & 
\multicolumn{1}{c}{$E(\theta,1,1000,2^{10},40)$} \\
\hline
    0.01 &            12.5094\\
    0.05 &            12.4565\\
    0.10 &            12.3877\\
    0.50 &            11.7406\\
    1.00 &            10.7505\\
    2.00 &             8.6679\\
    5.00 &             4.9654\\
   10.00 &             3.1174\\
   20.00 &             2.2010\\
   50.00 &             1.6947\\
  100.00 &             1.5370\\
\hline
\end{tabular}
\end{table}
In Table \ref{table:theta-var} we observe that the discretized difference of the approximate solutions to $(StAC)$ and $(SHE)$ in $L^2(\Omega\times (0,T)\times \Lambda)$, denoted by $E(\theta,1,1000,2^{10},40)$ decreases as $\theta>0$ increases. As larger values of $E(\theta,1,1000,2^{10},40)$ correspond with a larger difference between the values of $u^{StAC}$ and $u^{SHE}$, we may conclude that the projection effect becomes stronger with decreasing $\theta>0$. This is in line with the results of Subsection \ref{path} and our theoretical results in Subsection \ref{projection}.
\begin{figure}[htbp]
    \centering
    \includegraphics[width=0.8\linewidth]{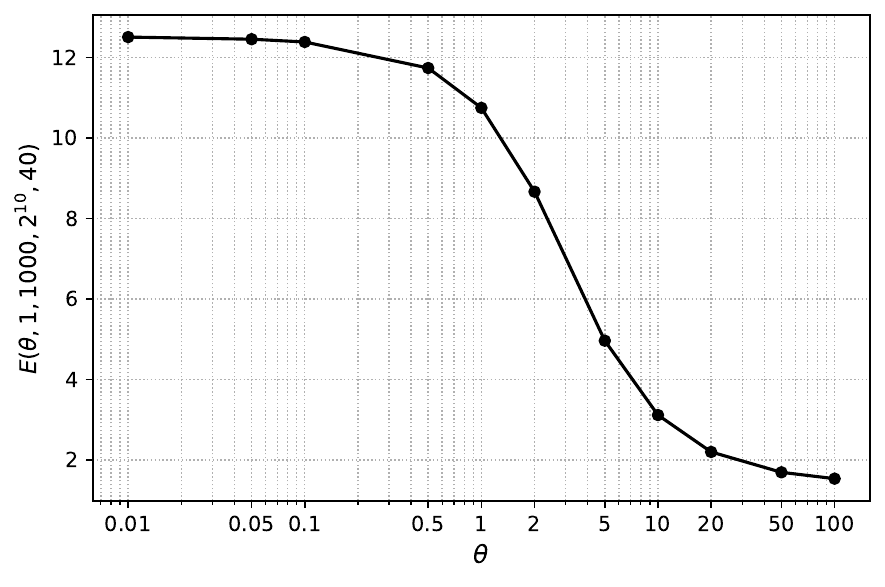}
    \caption{Dependence of 
    $E(\theta,1,1000,2^{10},40)$ on the parameter $\theta>0$.}
    \label{fig:theta-mc-error}
\end{figure}
In Figure \ref{fig:theta-mc-error}, we plot $E(\theta,1,1000,2^{10},40)$ as a function of $\theta>0$ on a $\log_{10}$-scale. The dependence on $\theta>0$ is clearly non-linear. For $\theta\in [0.01,1]$, the values of $E(\theta,1,1000,2^{10},40)$ decrease only slowly. In the intermediate range $\theta\in (1,10]$, we observe the most pronounced decay. For larger values $\theta\in (10,100]$ a saturation effect appears: further increases in $\theta$ lead to only a very mild additional decrease of $E(\theta,1,1000,2^{10},40)$. 
Consequently, the numerical experiments suggest choosing $0 < \theta \leq 1$ to obtain a strong projection effect. In the next experiment, we fix $\theta=1$ and study $E(1,C,1000,2^{10},40)$ with varying parameter $C\in \{0.1, 0.3, 0.5, 0.7, 0.9, 1\}$. The values are collected in Table \ref{table:C-var}, where we see that $E(1,C,1000,2^{10},40)$ is decreasing as $C>0$ is increasing. As larger values of $E(\theta,1,1000,2^{10},40)$ correspond with a larger difference between the values of $u^{StAC}$ and $u^{SHE}$, we may conclude that the projection effect improves with decreasing $C>0$. In Figure \ref{fig:C-mc-error} we plot $E(1,C,1000,2^{10},40)$ as a function of $C>0$ on a linear scale. We see that the dependence in $C>0$ is clearly linear. This numerical experiment supports the theoretical result from Subsection \ref{projection} by demonstrating that smaller values of $C>0$ lead to a stronger and more effective projection effect.
\begin{table}[htbp]
\centering
\caption{$E(1,C,1000,2^{10},40)$ for 
variable $C>0$.}
\label{table:C-var}
\small
\begin{tabular}{rr}
\multicolumn{1}{c}{$C$} & 
\multicolumn{1}{c}{$E(1,C,1000,2^{10},40)$} \\
\hline
 0.1 &             12.7857 \\
 0.3 &             12.3438 \\
 0.5 &             11.8669 \\
 0.7 &             11.3615 \\
 0.9 &             10.8357 \\
 1.0 &           10.5683 \\
\hline
\end{tabular}
\end{table}
\begin{figure}[htbp]
    \centering
    \includegraphics[width=0.8\linewidth]{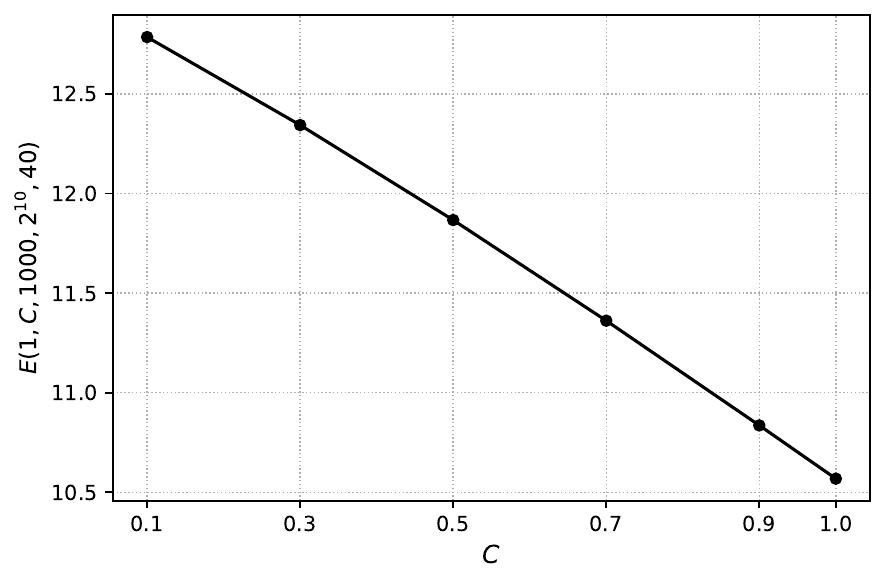}
    \caption{Dependence of 
    $E(1,C,1000,2^{10},40)$ on the parameter $C>0$.}
    \label{fig:C-mc-error}
\end{figure}
\subsection{Numerical error estimates}\label{experiments}
In this section we provide numerical error estimates in time of solutions of the splitting method. More precisely, if $u$ is as in Definition \ref{solution}, we are interested in numerical estimates of
\begin{align}\label{eq_090926_01}
\erww{\Vert u(T) - \tilde{u}^N \Vert_{L^2(\Lambda)}^2}
\end{align}
for $T,h, \theta >0$, $N \in \N$ and $\Delta t \in \mathcal{O}(\eps^{2+\theta})$.
To this end, we choose $T=C=\theta=1$, i.e., for $N \in \N$, we have $\eps = (1/N)^{1/3}$. Moreover, for $n \in \{0,...,N\}$ we set $\tilde{u}_{L,N,p}^n$ to be the solution of the splitting method for the $p$-th sample path, where $N, L, n$ and $\mathcal{P}$ are as in Section \ref{Study projection} and $p \in \{1,...,\mathcal{P}\}$.
Now, we fix quantities $L_{max} = 40$ and $N_{max} = 15360$ so that $\tilde{u}_{L_{max},N_{max},p}^{N_{max}}$ is a substitution of the exact function $u$ from Definition \ref{solution} for the $p$-th path. For $L \in \{14,20, 40\}$, $N \in \{256, 384, 512, 768, 1024, 1536\}$ and $\mathcal{P} = 5000$ we use the classical Monte Carlo method to calculate
\begin{align*}
    E(L,L_{max},N,N_{max}, \mathcal{P}) := \frac{1}{\mathcal{P}} \sum\limits_{p=1}^{\mathcal{P}}\Vert \tilde{u}_{L_{max},N_{max},p}^{N_{max}} - \tilde{u}_{L,N,p}^N \Vert_{L^2(\Lambda)}^2
\end{align*}
as an approximation for the term in \eqref{eq_090926_01}. The results together with a linear regression line are displayed in Figure \ref{fig:L=14}, \ref{fig:L=20} and \ref{fig:L=40} in a log-log coordinate system. The slope $\alpha$ of these regression lines turns out to be approximately equal to $1$ in all cases $L \in \{14, 20, 40\}$, i.e., we have approximately $E(L,L_{max},N,N_{max}) \in \mathcal{O}(\Delta t)$.
\begin{figure}[htbp]
    \centering
    \includegraphics[width=0.8\linewidth]{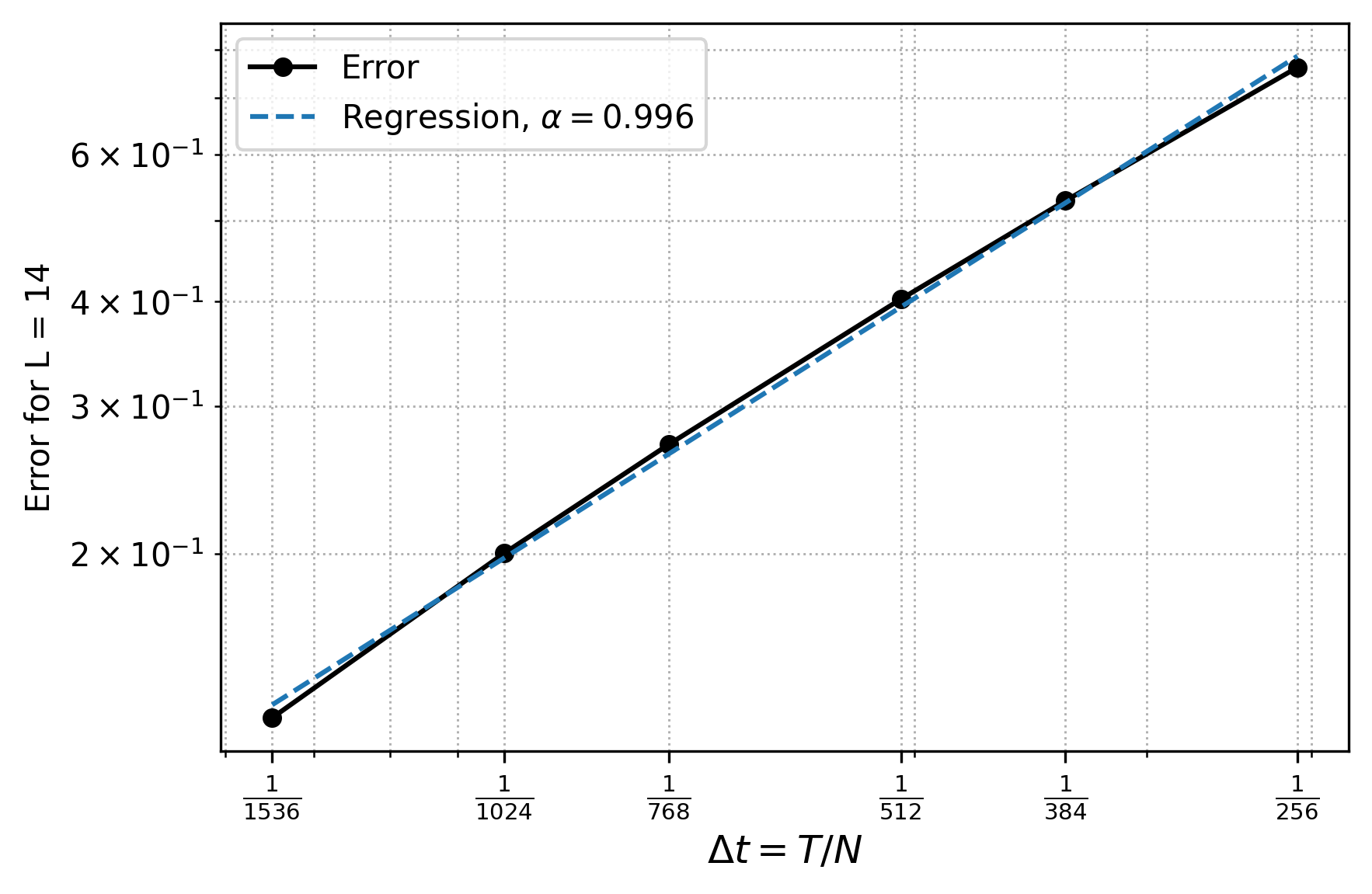}
    \caption{Error for $L=14$ and $N \in \{256, 384, 512, 768, 1024, 1536\}$.}
    \label{fig:L=14}
\end{figure}
\begin{figure}[htbp]
    \centering
    \includegraphics[width=0.8\linewidth]{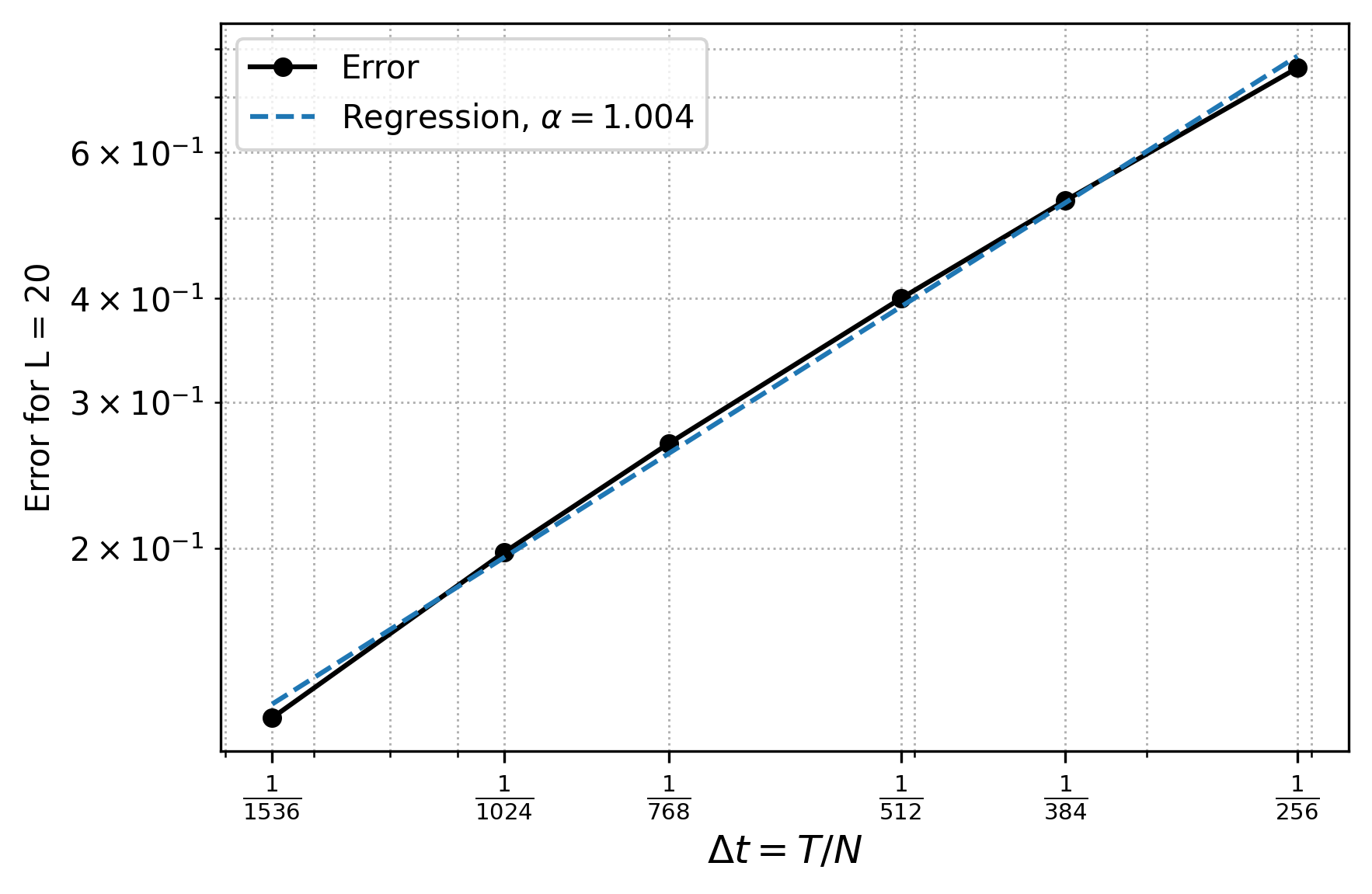}
    \caption{Error for $L=20$ and $N \in \{256, 384, 512, 768, 1024, 1536\}$.}
    \label{fig:L=20}
\end{figure}
\begin{figure}[htbp]
    \centering
    \includegraphics[width=0.8\linewidth]{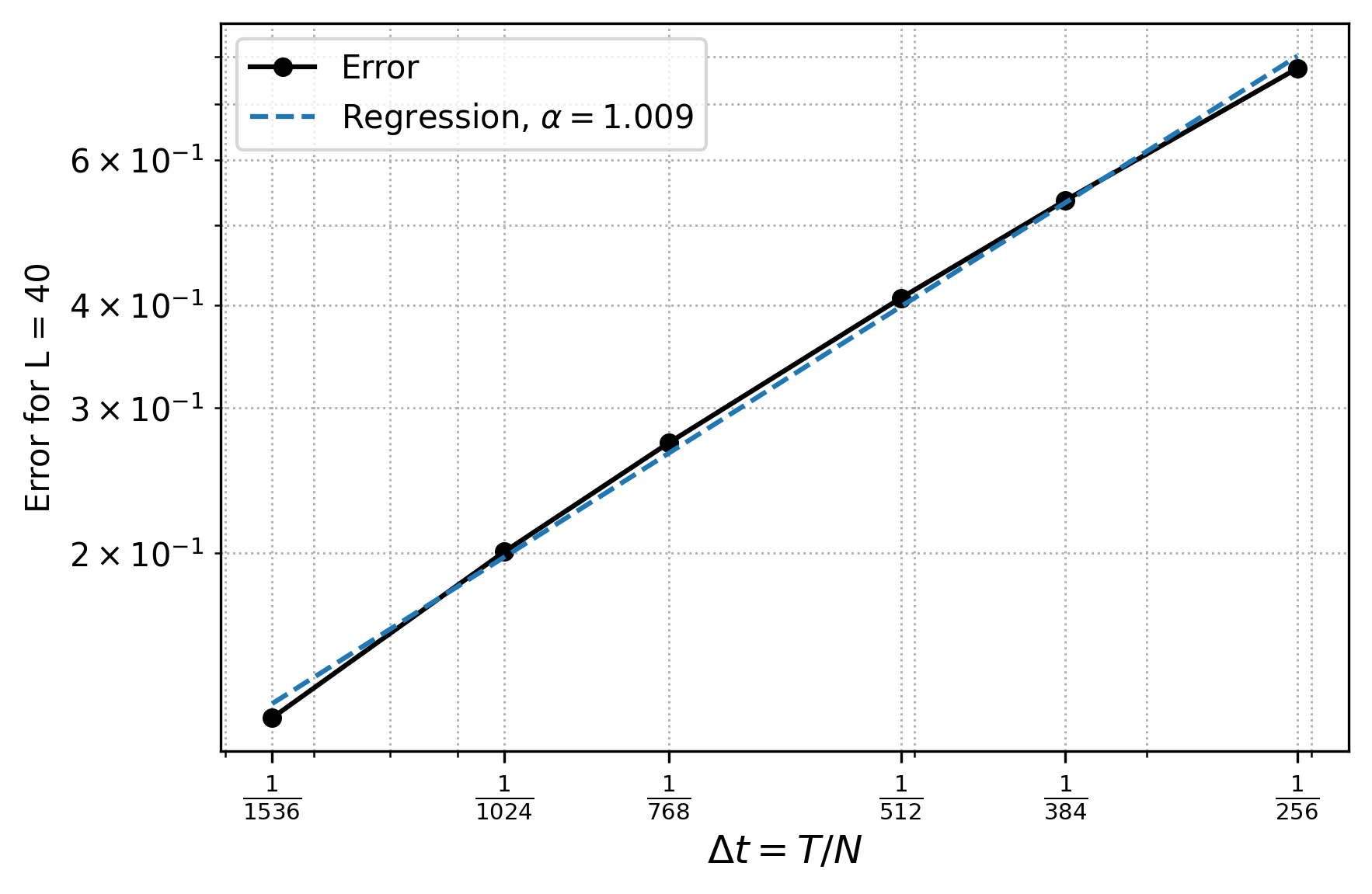}
    \caption{Error for $L=40$ and $N \in \{256, 384, 512, 768, 1024, 1536\}$.}
    \label{fig:L=40}
\end{figure}

\bibliographystyle{plain}
\bibliography{StAC}

\end{document}